%% file: final_article.tex
\documentclass[final,onefignum,onetabnum]{siamonline190516}

\usepackage{amsmath,amssymb,color,amsmath, graphicx, float, subcaption, mathrsfs,color} 
\usepackage{bm}
\usepackage{epstopdf }
\usepackage[font={small}]{caption}
\usepackage[font={footnotesize}]{subcaption}
\usepackage{cite}
\usepackage{url}

\usepackage{hyperref}

\usepackage{chngcntr}

\newtheorem{condition}[theorem]{\textbf{Condition}}

\input{ex_shared}

\ifpdf
\hypersetup{
  pdftitle={On the Convergence Rate of Projected Gradient Descent for a Back-Projection based Objective},
  pdfauthor={T. Tirer and R. Giryes}
}
\fi

\def\A{\bm{A}}

\def\D{\bm{D}}

\def\I{\bm{I}}

\def\P{\bm{P}}
\def\Q{\bm{Q}}

\def\U{\bm{U}}
\def\V{\bm{V}}
\def\W{\bm{W}}

\def\e{\bm{e}}

\def\g{\bm{g}}
\def\h{\bm{h}}

\def\u{\bm{u}}
\def\v{\bm{v}}

\def\x{\bm{x}}
\def\y{\bm{y}}
\def\z{\bm{z}}

\def\bLambda{\boldsymbol{\Lambda}}

\def\bxi{\bm{\xi}}

\def\infim#1{\underset{#1}{\mathrm{inf}}}

\def\argmin#1{\underset{#1}{\textrm{argmin}}}

\def\maxim#1{\underset{#1}{\textrm{max}}}
\def\supr#1{\underset{#1}{\mathrm{sup}}}
\def\minim#1{\underset{#1}{\mathrm{min}}}

\def\0{\mathbf{0}}
\def\1{\mathbf{1}}

\newcommand{\tomt}[1]{\textcolor{black}{#1}}

\begin{document}

\maketitle

\begin{abstract}
Ill-posed linear inverse problems appear in many scientific setups,
and are typically addressed by solving optimization problems, which are composed of data fidelity and prior terms. 
Recently, several works have considered a back-projection (BP) based fidelity term as an alternative to the common least squares (LS), and demonstrated excellent results for popular inverse problems. 
These works have also empirically shown that using the BP term, rather than the LS term, requires fewer iterations of 
optimization algorithms. 
In this paper, we examine the convergence rate of the projected gradient descent (PGD) algorithm for the BP objective. 
Our analysis allows to identify an inherent source for its faster convergence compared to using the LS objective, while making only mild assumptions. 
We also analyze the more general proximal gradient method under a relaxed contraction condition on the proximal mapping of the prior. This analysis further highlights the advantage of BP when the linear measurement operator is badly conditioned. 
Numerical experiments with both $\ell_1$-norm and GAN-based priors 
corroborate our theoretical results. 

\end{abstract}




\section{Introduction}
\label{sec:intro}

The task of recovering a signal from its observations that are obtained by some acquisition process is common in many fields of science and engineering, and referred to as an {\em inverse problem}. 
In imaging science, the inverse problems are often linear, in the sense that the observations can be formulated by a linear model
\begin{align}
\label{Eq_general_model}
\y = \A\x_{gt} + \e,
\end{align}
where $\tomt{\x_{gt}} \in \mathbb{R}^n$ represents the unknown 
\tomt{ground truth} 
image, $\y \in \mathbb{R}^m$ represents the observations, $\A$ is an $m \times n$ measurement matrix, 
$\e \in \mathbb{R}^m$ is a noise vector, 
and typically $m \leq n$. 
For example, this model corresponds to tasks like denoising \cite{rudin1992nonlinear,donoho1995noising,dabov2007image}, deblurring \cite{biemond1990iterative,danielyan2012bm3d},  super-resolution \cite{sun2008image,yang2010image}, and compressed sensing \cite{donoho2006compressed, candes2008introduction}. 

A common strategy for recovering $\x_{gt}$ is to solve an optimization problem, which is composed of a fidelity term $\ell(\cdot)$ that enforces agreement with the observations $\y$, and a prior term $s(\cdot)$, 
which is inevitable as the inverse problems represented by \eqref{Eq_general_model} are usually ill-posed (i.e., the measurements do not suffice for obtaining a successful reconstruction). 
The optimization problem is usually stated in a penalized form
\begin{align}
\label{Eq_cost_func_general}
\minim{{\x}} \,\, \ell({\x}) + \beta s({\x}),
\end{align}
or in a constrained form
\begin{align}
\label{Eq_cost_func_general2}
\minim{{\x}} \, \ell({\x}) \,\,\,\,\, \mathrm{s.t.} \,\,\,\, s({\x}) \leq R,
\end{align}
where $\beta$ and $R$ are positive scalars that control the regularization level  
and ${\x}$ is the optimization variable.

While a vast amount of research has focused on designing good prior models, most of the works use a typical least squares (LS) fidelity term 
\begin{align}
\label{Eq_fidelity_typical}
\ell_{LS}({\x}) \triangleq  \frac{1}{2} \| \y-\A{\x} \|_2^2,
\end{align}
where $\| \cdot \|_2$ stands for the Euclidean norm. 
Using the LS term is common perhaps because it can be derived from the negative log-likelihood function under the assumption of white Gaussian noise. 
However, even under this assumption, note that in general the maximum likelihood estimation has optimality properties only when the number of measurements is much larger than
the number of unknown variables, {\em which is obviously not the case in ill-posed problems}.

Recently, a different fidelity term, 
dubbed as the ``back-projection" (BP) term, has been 
identified and studied 
\cite{tirer2019back}.
Assuming that $m \leq n$ and $\mathrm{rank}(\A)=m$ (which is the common case, e.g., in super-resolution and compressed sensing tasks), this term 
can be written as 
\begin{align}
\label{Eq_fidelity_idbp}
\ell_{BP}({\x}) \triangleq  \frac{1}{2} \| \A^\dagger ( \y - \A {\x})  \|_2^2,
\end{align}
where $\A^\dagger \triangleq \A^T(\A\A^T)^{-1}$ is the pseudoinverse of $\A$, 
or equivalently\footnote{\tomt{The equivalence follows from the identities $(\A^\dagger)^T \A^\dagger = (\A\A^T)^{-1}$ and $\A^\dagger\A=(\A^\dagger\A)^T=(\A^\dagger\A)^2$, and the expansion of the two quadratic forms: $ \| \A^\dagger ( \y - \A \x)  \|_2^2 = \| \A^\dagger\y \|_2^2 -2\y^T(\A^\dagger)^T\A^\dagger\A\x + \| \A^\dagger\A\x \|_2^2 = \y^T(\A\A^T)^{-1}\y -2\y^T(\A\A^T)^{-1}\A\x +\x^T\A^T(\A\A^T)^{-1}\A\x = \| (\A\A^T)^{-\frac{1}{2}} ( \y - \A \x ) \|_2^2$.}}
as
\begin{align}
\label{Eq_fidelity_idbp2}
\ell_{BP}({\x}) =  \frac{1}{2} \| (\A\A^T)^{-\frac{1}{2}} ( \y - \A {\x} ) \|_2^2.
\end{align}

The BP fidelity term has been {\em implicitly} used in the IDBP framework
\cite{tirer2019image}, where it has been combined with plug-and-play denoisers such as BM3D \cite{dabov2007image} and DnCNN \cite{zhang2017beyond} and demonstrated state-of-the-art reconstruction results for image super-resolution \cite{tirer2018super} and deblurring \cite{tirer2019image, tirer2018icip}. 
The BP term also {\em implicitly} relates to the compressed sensing method in \cite{liao2014generalized,yuan2021plug}. 
The {\em explicit} connection of previous works to \eqref{Eq_fidelity_idbp} has been 
pointed out 
in \cite{tirer2019back} and follows from applying the proximal gradient method on $\ell_{BP}({\x})+\beta s({\x})$. 

\tomt{
The work in \cite{tirer2019back} has focused on examining and comparing the LS and BP terms from an estimation accuracy point of view. By mathematically analyzing the cost functions for the Tikhonov regularization prior, and empirically studying more sophisticated priors, it has identified cases (such as tasks for which $\A\A^T$ is badly conditioned) where the BP term yields reconstructions with better mean squared error (MSE) than the LS term.} 

\tomt{
To intuitively understand why using the BP term can yield a recovery with better MSE, observe the following. Let $\A=\U \bLambda \V^T$ be the singular value decomposition (SVD) of $\A$, namely, $\U \in \mathbb{R}^{m \times m}$ and $\V \in \mathbb{R}^{n \times n}$ are orthogonal matrices and $\bLambda$ is an $m \times n$ rectangular diagonal matrix with nonzero singular values $\{ \lambda_i \}_{i=1}^m$ on the diagonal. In the noiseless case ($\y=\A\x_{gt}$), it can be shown (see \cite{tirer2019back}) that 
$\ell_{LS}({\x}) = \frac{1}{2} \sum_{i=1}^m \lambda_i^2 | \v_i^T({\x}-\x_{gt})|^2$ 
and
$\ell_{BP}({\x}) = \frac{1}{2} \sum_{i=1}^m | \v_i^T({\x}-\x_{gt})|^2$,
where $\v_i$ is the $i$th column of $\V$. Observe that $\ell_{BP}({\x})$ equally weighs all the error components $\{ | \v_i^T({\x}-\x_{gt})|^2 \}_{i=1}^m$, contrary to $\ell_{LS}({\x})$, which weighs them according to $\{ \lambda_i^2 \}_{i=1}^m$. Now, note the similarity between $\ell_{BP}({\x})$ and the MSE, which can be formulated as $\| {\x} - \x_{gt} \|_2^2 = \sum_{i=1}^n | \v_i^T({\x}-\x_{gt})|^2 $ (note that the sum here goes over all the $n$ basis vectors in $\V$). Clearly, 
the error components are equally weighted in the MSE (as in BP). 
In the noisy case, a more intricate analysis is required \cite{tirer2019back}. Interestingly, a recent paper \cite{abu2021image} has shown also a connection between $\ell_{BP}({\x})$ and 
an {\em estimator} of the MSE (independent of $\x_{gt}$), namely, an adaptation of Stein's unbiased risk estimate (SURE) \cite{Stein:1981vf} to linear ill-posed problems \cite{eldar2008generalized}.}

Recent applications of the explicit BP term include 
\cite{sabulal2020joint,zukerman2020bp,yogev2021interpretation}. 
The BP term is also related to ALISTA \cite{liu2019alista}, which is very similar to IDBP with $s(\cdot)$ being the $\ell_1$-norm. 
ALISTA essentially enjoys the benefits of the BP term 
to get a good initial state that already converges fast, and accelerates it by learning the step-sizes and the soft-thresholds (more details in Appendix~\ref{app:alista}).

Empirical results in previous works have shown that using the BP term, rather than the LS term, {\em requires fewer iterations} of optimization algorithms, such as  
proximal gradient methods. 
This implies reduced overall run-time when
the operator $\A^\dagger$ has fast implementation (e.g., in  
deblurring and super-resolution) 
or if the proximal operation dominates the computational cost of each iteration. 
We emphasize that this convergence advantage of BP over LS has not been mathematically analyzed in prior works.

\textbf{Contribution.} 
In this paper, we provide mathematical reasoning for the faster convergence of BP compared to LS, for both projected gradient descent (PGD) applied on the constrained form \eqref{Eq_cost_func_general2}, and the more general proximal gradient method 
applied on the penalized form \eqref{Eq_cost_func_general}.
Our analysis for PGD (Section~\ref{sec:convergence}), which is inspired by the analysis in \cite{oymak2017sharp}, requires very mild assumptions and allows us to identify sources for the different convergence rates.
Our analysis for proximal methods (Section \ref{sec:convergence_beyond}) 
requires a relaxed contraction condition on the proximal mapping of the prior under which it 
further highlights the advantage of BP when $\A\A^T$ is badly conditioned. 
Numerical experiments (Section~\ref{sec:exp}) 
corroborate our theoretical results for PGD with both convex ($\ell_1$-norm) and non-convex (pre-trained DCGAN \cite{radford2015unsupervised}) priors.
For the $\ell_1$-norm prior, we also present experiments for proximal methods, and connect them with our analysis. 

Importantly, notice that using the BP term 
rather than the LS is fundamentally different than the technique of preconditioning, where the optimization objective and minimizers {\em are not modified} and existing acceleration results require strong convexity of the objective, which is not the case here 
(more details in Appendix~\ref{app:precond}).


\section{Preliminaries}
\label{sec:prelim}

Let us 
present notations and definitions 
that are used in the paper. 
We write $\|\cdot\|_2$ for the Euclidean norm of a vector, $\|\cdot\|$ for the spectral norm of a matrix, and $\sigma_{max}(\cdot)$ and $\sigma_{min}(\cdot)$ for the largest and smallest eigenvalue of a matrix, respectively.
We denote the unit Euclidean ball and sphere in $\mathbb{R}^n$ by $\mathbb{B}^n$ and $\mathbb{S}^{n-1}$, respectively.
We denote by \tomt{$\mathcal{P}_{\mathcal{K}}(\z)$ the Euclidean projection of $\z$ onto the set $\mathcal{K}$}. 
We denote by $\I_n$ the identity matrix in $\mathbb{R}^n$, and 
by $\P_{A} \triangleq \A^\dagger \A$ and $\Q_{A} \triangleq \I_n - \P_{A}$ the \tomt{orthogonal projection matrices that project vectors from $\mathbb{R}^n$} onto the row space  
and the null space (respectively) of 
the full row-rank matrix $\A$. 
Let us also define the descent set and its tangent cone \cite{chandrasekaran2012convex} as follows.

\begin{definition}
\label{def:DnC}
The descent set of the function $s$ 
at a point $\x_{*}$ is defined as
\begin{align}
\mathcal{D}_{s}(\x_{*}) \triangleq \{ \h \in \mathbb{R}^n : s(\x_{*} + \h) \leq s(\x_{*}) \}.
\end{align}
The tangent cone $\mathcal{C}_{s}(\x_{*})$ at a point $\x_{*}$ is the smallest closed cone satisfying $\mathcal{D}_{s}(\x_{*}) \subseteq \mathcal{C}_{s}(\x_{*})$.
\end{definition}

In this paper, we largely focus on minimizing \eqref{Eq_cost_func_general2} using PGD, i.e., by applying iterations of the form
\begin{align}
\label{Eq_pgd_iter}
{\x}_{t+1} = \mathcal{P}_{\mathcal{K}} \left ( {\x}_{t} - \mu \nabla \ell({\x}_{t}) \right ),
\end{align}
where $\nabla\ell({\x})$ is the gradient of $\ell(\cdot)$ at ${\x}$, $\mu$ is a step-size, and 
\begin{align}
\label{Eq_set_Kr}
\mathcal{K} \triangleq \left \{ {\x} \in \mathbb{R}^n : s({\x}) \leq R \right \}.
\end{align}
Note that 
\begin{align}
\label{Eq_fidelity_grads}
\nabla \ell_{LS}({\x}) &= \A^T ( \A {\x} - \y ), \nonumber \\
\nabla \ell_{BP}({\x}) &= \A^\dagger ( \A {\x} - \y ).
\end{align}
Therefore, we can examine a unified formulation of PGD for both objectives
\begin{align}
\label{Eq_pgd_iter2}
{\x}_{t+1} = \mathcal{P}_{\mathcal{K}} \left ( {\x}_{t} + \mu \W ( \y - \A {\x}_t ) \right ),
\end{align}
where $\W$ equals $\A^T$ or $\A^\dagger$ for the LS and BP terms, respectively.


\section{Comparing PGD Convergence Rates}
\label{sec:convergence}

The goal of this section is to provide a mathematical reasoning for the observation (shown in Section \ref{sec:exp}) that 
using the BP term, rather than the LS term, requires fewer PGD iterations. 
We start in Section~\ref{sec:toy} with a warm-up example with a very restrictive prior that fixes the value of ${\x}$ on the null space of $\A$, which provides us with some intuition as to the advantage of BP. 
Then, in Sections \ref{sec:general} - \ref{sec:convergence_ext} we build on the analysis technique in \cite{oymak2017sharp} to show that the advantage of BP carries on to practical priors.

\subsection{Warm-Up: Restrictive ``Oracle'' Prior}
\label{sec:toy}

Let us define the following 
``oracle''\footnote{In fact, the results in this warm-up require that the prior fixes $\Q_{A}{\x}$ to a constant value on the null space of $\A$, but the value itself does not affect the convergence rates.} prior that fixes the value of ${\x}$ on the null space of $\A$ to that of the latent $\x_{gt}$
\begin{equation}
\label{Eq_oracle_prior}
s_{oracle}({\x})=
\begin{cases} 
      0, & {\x}: \Q_{A}{\x}=\Q_{A}\x_{gt} \\
      +\infty,  & otherwise
   \end{cases}.
\end{equation}
Applying the PGD update rule from \eqref{Eq_pgd_iter2} using this prior, we have
\begin{equation}
\label{Eq_oracle_pgd}
{\x}_{t+1} = \P_{A} \left ( {\x}_{t} + \mu \W ( \y - \A {\x_t} ) \right ) + \Q_{A}\x_{gt}.  
\end{equation}
In the following, we specialize \eqref{Eq_oracle_pgd} for LS and BP with 
step-size of 1 over the 
Lipschitz constant of $\nabla \ell(\cdot)$. 
This step-size 
is perhaps the most common choice of practitioners, as it ensures (sublinear) convergence of the sequence $\{ {\x}_{t} \}$ for general convex priors \cite{beck2009fast} (i.e., for larger {\em constant} step-size, PGD and general proximal methods may ``swing" and not converge). 
Detailed explanation for the popularity of this constant step-size is given in Appendix~\ref{app:StepSize}. 
Here, due to the constant Hessian matrix $\nabla^2 \ell$ for LS and BP, this step-size can be computed as $\| \nabla^2 \ell \|^{-1}$.

\textbf{LS case:} 
For the LS objective, we have $\W=\A^T$ and $\mu_{LS} =  \| \nabla^2 \ell_{LS}\|^{-1} = \| \A^T\A \|^{-1} = 1/\sigma_{max}(\A\A^T)$.
So, 
\begin{align}
{\x}_{t+1}^{LS} &= \P_{A} \left ( {\x}_{t}^{LS} + \mu_{LS} \A^T ( \y - \A {\x}_t^{LS} ) \right ) + \Q_{A}\x_{gt}  \nonumber \\
&= \P_{A} \left ( ( \I_n - \mu_{LS} \A^T \A ) {\x}_{t}^{LS} + \mu_{LS} \A^T \y \right ) + \Q_{A}\x_{gt}  \nonumber \\
&= ( \P_{A} - \mu_{LS} \A^T \A ) {\x}_{t}^{LS} + \mu_{LS} \A^T \y + \Q_{A}\x_{gt}. 
\end{align}
Let $\x_*^{LS}$ be the stationary point of the sequence $\{ {\x}_{t}^{LS} \}$, i.e., $\x_*^{LS} = ( \P_{A} - \mu_{LS} \A^T \A ) \x_*^{LS} + \mu_{LS} \A^T \y + \Q_{A}\x_{gt}$.
The convergence rate can be obtained as follows
\begin{align}
\| {\x}_{t+1}^{LS} - \x_*^{LS} \|_2 &= \|  ( \P_{A} - \mu_{LS} \A^T \A ) ({\x}_{t}^{LS} - \x_*^{LS} ) \|_2 \nonumber \\
& \leq \left (1- \frac{\sigma_{min}(\A\A^T)}{\sigma_{max}(\A\A^T)} \right ) \| {\x}_{t}^{LS} - \x_*^{LS}  \|_2. 
\end{align}

\textbf{BP case:} 
For the BP objective, we have $\W=\A^\dagger$ and $\mu_{BP} =  \| \nabla^2 \ell_{BP} \|^{-1} = \| \A^\dagger\A \|^{-1}=1$, where the last equality follows from the fact that $\P_{A}=\A^\dagger\A$ is a non-trivial orthogonal projection. 
Substituting these terms in \eqref{Eq_oracle_pgd}, we get
\begin{align}
{\x}_{t+1}^{BP} &= \P_{A} \left ( {\x}_{t}^{BP} + \A^\dagger ( \y - \A {\x}_t^{BP} ) \right ) + \Q_{A}\x_{gt}  \nonumber \\
&= \P_{A} \left ( \Q_{A} {\x}_{t}^{BP} + \A^\dagger \y \right ) + \Q_{A}\x_{gt}  \nonumber \\
&= \A^\dagger \y + \Q_{A}\x_{gt}. 
\end{align}

Note that while the use of LS objective leads to linear convergence rate of $1- \frac{\sigma_{min}(\A\A^T)}{\sigma_{max}(\A\A^T)}$, using BP objective requires only a {\em single} iteration. 
This result hints that an advantage of BP may exist even for 
practical priors $s({\x})$, which only implicitly impose some restrictions on $\Q_{A}{\x}$.

\subsection{General Analysis}
\label{sec:general}

The following theorem provides a term that characterizes the convergence rate of PGD for both LS and BP objectives 
for general priors. 
It is closely related to Theorem 2 in \cite{oymak2017sharp}. The difference is twofold. First, the theorem in \cite{oymak2017sharp} considers only the LS objective and its derivation is not valid for the BP objective. 
Second, as the authors of \cite{oymak2017sharp} focus on the estimation error, they examine 
$\|{\x}_{t}-\x_{gt}\|_2$, where $\x_{gt}$ is the unknown ground truth signal, and assume 
that $s(\x_{gt})$ is known, which allows to set $R=s(\x_{gt})$. 
In contrast, we generalize the theory for both LS and BP objectives, and for an arbitrary value of $R$. 
Among others, 
our theorem covers any stationary point $\x_*$ of the PGD scheme \eqref{Eq_pgd_iter2} (i.e., an optimal point for convex $s(\cdot)$) for which $s(\x_*)=R$.\footnote{Essentially, we require that $R$ is small enough such that the prior is not meaningless.}
The proofs of the theorem and its following propositions are deferred to Appendix~\ref{app:proofs}.

\begin{theorem}
\label{theorem1}
Let $s : \mathbb{R}^n \rightarrow \mathbb{R}$ be a lower semi-continuous function, 
and
let $\x_*$ be a point on the boundary of $\mathcal{K}$, i.e., $s(\x_*)=R$.
Let $\kappa_s$ be a constant that is equal to 1 for convex $s$ and equal to 2 otherwise. 
Then, the sequence $\{{\x}_t\}$ obtained by 
\eqref{Eq_pgd_iter2} obeys
\begin{equation}
\label{Eq_pgd_thm}
\|{\x}_{t+1}-\x_*\|_2 \leq \kappa_s \rho(\mathcal{C}_{s}(\x_{*})) \|{\x}_{t}-\x_*\|_2 + \kappa_s \mu \xi(\mathcal{C}_{s}(\x_{*})),
\end{equation}
where
\begin{align}
\label{Eq_pgd_thm_def}
\rho(\mathcal{C}_{s}(\x_{*})) &\triangleq \supr{ \bm{u},\bm{v} \in \mathcal{C}_{s}(\x_{*}) \cap \mathbb{B}^n } \bm{u}^T (\I_n - \mu \W \A) \bm{v}, \nonumber \\
\xi(\mathcal{C}_{s}(\x_{*})) &\triangleq \supr{ \bm{v} \in \mathcal{C}_{s}(\x_{*}) \cap \mathbb{B}^n } \bm{v}^T \W (\y - \A\x_*).
\end{align}
\end{theorem}

When $\kappa_s \rho(\mathcal{C}_{s}(\x_{*})) < 1$, 
Theorem~\ref{theorem1} 
implies linear convergence 
\tomt{(up to an error term that can be eliminated in certain settings)} 
and provides characterization of its rate.
\tomt{
The key for obtaining linear convergence, without general strong convexity of the problem, is the restriction that the prior imposes on the signal 
in the null-space of $\A$.
This restriction is captured by lower bounding above zero the ``restricted smallest eigenvalue" of $\A^T\A$ and $\A^\dagger\A$ (for LS and BP terms, respectively), which takes into account the prior --- it is being searched for only in $\mathcal{C}_{s}(\x_{*})$ (the tangent cone of the prior's descent set at $\x_{*}$, recall Definition~\ref{def:DnC}).} 
We elaborate on this separately 
for LS and BP, 
below Propositions \ref{prop_ls} and \ref{prop_bp}, respectively.

Assuming that $\kappa_s \rho(\mathcal{C}_{s}(\x_{*})) < 1$,  
the term $\xi(\mathcal{C}_{s}(\x_{*}))$ 
belongs to the component of the bound \eqref{Eq_pgd_thm} 
which cannot be compensated for by using more iterations. 
Note that if $R=s(\x_{gt})$, then 
Theorem~\ref{theorem1} can be applied with $\x_*=\x_{gt}$. 
In this case, $\xi(\mathcal{C}_{s}(\x_{*}))=\xi(\mathcal{C}_{s}(\x_{gt}))$ characterizes the estimation error $\lim_{t\to\infty} \|{\x}_{t}-\x_{gt}\|_2$ (up to a factor due to the recursion in \eqref{Eq_pgd_thm}). 
Moreover, $\y - \A\x_*=\y - \A\x_{gt}=\e$, so 
the term $\xi(\mathcal{C}_{s}(\x_{gt}))$ 
vanishes if there is no noise.
That is, for $R=s(\x_{gt})$ and no noise, we have 
\begin{align}
\label{Eq_pgd_thm_noiseless}
\|{\x}_{t+1}-\x_{gt}\|_2 \leq \kappa_s \rho(\mathcal{C}_{s}(\x_{gt})) \|{\x}_{t}-\x_{gt}\|_2.
\end{align}

In 
practice, one typically does not know the value of $s(\x_{gt})$ and often $R$ that is not equal to $s(\x_{gt})$ provides better results in the presence of noise 
or when $s(\cdot)$ is non-convex.
Therefore, in this work we aim  
to compare the convergence rates for LS and BP objectives for arbitrary values of $R$. 

As we consider arbitrary $R$, we focus 
on $\x_*=\lim_{t\to\infty}{\x}_{t}$, i.e., the stationary point obtained by PGD. In this case $\lim_{t\to\infty} \|{\x}_{t}-\x_*\|_2=0$, and thus the component with $\xi(\mathcal{C}_{s}(\x_{*}))$ in \eqref{Eq_pgd_thm} presents slackness that is a consequence of the proof technique. 
To further see that $\xi(\mathcal{C}_{s}(\x_{*}))$ is not expected to affect the conclusions of our analysis,  
note that we examine PGD with step-sizes that ensure convergence in convex settings \cite{beck2009fast} (namely, with the common step-size of $\| \nabla^2 \ell \|^{-1}$). Therefore, for convex $s(\cdot)$ misbehavior of $\{{\x}_{t}\}$ like ``swinging'' is not possible.  
Empirically, monotonic convergence of $\{{\x}_{t}\}$ is observed in Section~\ref{sec:exp} even for highly non-convex prior such as DCGAN.

In the rest of this section we focus 
on the term $\rho(\mathcal{C}_{s}(\x_{*}))$ in \eqref{Eq_pgd_thm}. 
Whenever $\kappa_s \rho(\mathcal{C}_{s}(\x_{*})) < 1$, 
this term characterizes the convergence rate of PGD: {\em smaller $\rho$ implies faster convergence}. 
We start with specializing and bounding it for $\ell_{LS}({\x})$ and $\ell_{BP}({\x})$. 

\begin{proposition}
\label{prop_ls}
Consider the LS objective $\ell_{LS}({\x})$ and step-size $\mu_{LS} \triangleq  \| \nabla^2 \ell_{LS} \|^{-1}$. We have
\begin{align}
\label{Eq_pgd_rate_ls}
\rho(\mathcal{C}_{s}(\x_{*})) & \leq 
1 - \frac{1}{\| \A^T\A \|} \infim{ \bm{u} \in \mathcal{C}_{s}(\x_{*}) \cap \mathbb{S}^{n-1} } \| \A \bm{u} \|_2^2 \nonumber \\
& \triangleq P_{LS}(\mathcal{C}_{s}(\x_{*})).
\end{align}
\end{proposition}

Various works \cite{chandrasekaran2012convex, plan2012robust, amelunxen2014living, genzel2017ell} 
have proved, via Gordon's lemma (Corollary 1.2 in \cite{gordon1988milman})  and the notion of {\em Gaussian width}, that if: 1) the entries of $\A \in \mathbb{R}^{m \times n}$ are i.i.d Gaussians $\mathcal{N}(0,\frac{1}{m})$;  
2) $\x_*$ belongs to a parsimonious signal model (e.g., a sparse signal); and 3) $s(\cdot)$ is an appropriate prior for the signal model (e.g.,~$\ell_0$-quasi-norm or $\ell_1$-norm for sparse signals), then there exist tight lower bounds\footnote{The tightness of these bounds has been shown empirically.} on the restricted smallest eigenvalue of $\A^T\A$: $\infim{ \bm{u} \in \mathcal{C}_{s}(\x_{*}) \cap \mathbb{S}^{n-1} } \|\A \bm{u} \|_2^2$, which are much greater than the naive lower bound $\sigma_{min}(\A^T\A) \| \bm{u} \|_2^2=0$ (recall that $m<n$, so $\sigma_{min}(\A^T\A)=0$). 
This implies that $\kappa_s P_{LS}(\mathcal{C}_{s}(\x_{*}))<1$ and therefore Theorem~\ref{theorem1} indeed provides meaningful guarantees for PGD applied on LS objective under the above conditions.

\begin{proposition}
\label{prop_bp}
Consider the BP objective $\ell_{BP}({\x})$ and step-size $\mu_{BP} \triangleq  \| \nabla^2 \ell_{BP} \|^{-1}$. We have
\begin{align}
\label{Eq_pgd_rate_bp}
\rho(\mathcal{C}_{s}(\x_{*})) & \leq 
1 - \infim{ \bm{u} \in \mathcal{C}_{s}(\x_{*}) \cap \mathbb{S}^{n-1} } \| (\A\A^T)^{-\frac{1}{2}}\A \bm{u} \|_2^2 \nonumber \\
& \triangleq P_{BP}(\mathcal{C}_{s}(\x_{*})).
\end{align}
\end{proposition}

As will be shown in Proposition~\ref{prop1} below, if $P_{LS}(\mathcal{C}_{s}(\x_{*}))<1$ then $P_{BP}(\mathcal{C}_{s}(\x_{*}))<1$ as well. Therefore, Theorem~\ref{theorem1} provides meaningful guarantees also for PGD applied on BP objective. 
However, obtaining tight lower bounds directly on the restricted smallest eigenvalue $\infim{ \bm{u} \in \mathcal{C}_{s}(\x_{*}) \cap \mathbb{S}^{n-1} } \|(\A\A^T)^{-\frac{1}{2}}\A \bm{u} \|_2^2$, similar to those obtained (in some cases) for $\infim{ \bm{u} \in \mathcal{C}_{s}(\x_{*}) \cap \mathbb{S}^{n-1} } \|\A \bm{u} \|_2^2$, appears to be an open problem.  
Its difficulty stems from the fact that 
tools like Slepian's lemma and Sudakov-Fernique inequality, which are the core of Gordon's lemma that is used to bound $\infim{ \bm{u} \in \mathcal{C}_{s}(\x_{*}) \cap \mathbb{S}^{n-1} } \|\A \bm{u} \|_2^2$, cannot be used in this case.

Denote by $\x_{*}^{LS}$ and $\x_{*}^{BP}$ the recoveries obtained by LS and BP objectives, respectively.
The terms $P_{LS}(\mathcal{C}_{s}(\x_{*}^{LS}))$ and $P_{BP}(\mathcal{C}_{s}(\x_{*}^{BP}))$ 
upper bound the convergence rate $\rho$ for each objective.
Observing these expressions, we identify two factors that affect their relation, and are thus possible sources for different convergence rates. The two factors, labeled as ``intrinsic'' and ``extrinsic'', are explained in Sections \ref{sec:convergence_intrin} and \ref{sec:convergence_ext}, respectively.

\subsection{Intrinsic Source of Faster Convergence for BP}
\label{sec:convergence_intrin}

\tomt{
Consider the case where the obtained minimizers are similar, i.e.,  $\x_{*}^{LS} \approx \x_{*}^{BP}$. 
The following proposition guarantees that $P_{BP}(\mathcal{C}_{s}(\x_{*}))$ is lower than $P_{LS}(\mathcal{C}_{s}(\x_{*}))$ for {\em any} full row-rank $\A$, 
which is an inherent advantage of the BP term. 
We believe that this advantage of the convergence rate of BP holds 
also when the two stationary points, $\x_{*}^{LS}$ and $\x_{*}^{BP}$, are not identical but rather similar or share similar geometry for their associated cones $\mathcal{C}_{s}$.}

\begin{proposition}
\label{prop1}
Consider the definitions in \eqref{Eq_pgd_rate_ls} and \eqref{Eq_pgd_rate_bp}.
We have 
$P_{BP}(\mathcal{C}_{s}(\x_{*})) \leq P_{LS}(\mathcal{C}_{s}(\x_{*}))$.
\end{proposition}

\begin{proof}
\begin{align}
P_{BP}(\mathcal{C}_{s}(\x_{*})) &= 1 - \infim{ \bm{u} \in \mathcal{C}_{s}(\x_{*}) \cap \mathbb{S}^{n-1} } \| (\A\A^T)^{-\frac{1}{2}}\A \bm{u} \|_2^2  \nonumber \\
&\leq 1 - \sigma_{min}( (\A\A^T)^{-1} ) \infim{ \bm{u} \in \mathcal{C}_{s}(\x_{*}) \cap \mathbb{S}^{n-1} } \|\A \bm{u} \|_2^2  \nonumber \\
&= 1 - \frac{1}{\|\A^T\A\|} \infim{ \bm{u} \in \mathcal{C}_{s}(\x_{*}) \cap \mathbb{S}^{n-1} } \|\A \bm{u} \|_2^2  
= P_{LS}(\mathcal{C}_{s}(\x_{*})) 
\end{align}
\end{proof}

Notice that in 
the last proof 
we use an inequality that does not take into account the fact that $\u$ resides in a restricted set. 
As discussed above, this is due to the lack of tighter lower bounds for $\infim{ \bm{u} \in \mathcal{C}_{s}(\x_{*}) \cap \mathbb{S}^{n-1} } \|(\A\A^T)^{-\frac{1}{2}}\A \bm{u} \|_2^2$.
Still, following the warm-up example\footnote{Note that the general analysis subsumes the warm-up result: {\em strict} inequality for the convergence rates. 
For the prior in \eqref{Eq_oracle_prior} we have that the descent set (and its tangent cone) are the subspace spanned by the rows of $\A$. Therefore, we have that $P_{LS}=1- \frac{\sigma_{min}(\A\A^T)}{{\|\A^T\A\|}}$ and $P_{BP}=1-\|\P_A\|=0$.} and the discussions below Propositions \ref{prop_ls} and \ref{prop_bp}, we conjecture 
that the inequality in Proposition \ref{prop1} is strict, i.e., that $P_{BP}(\mathcal{C}_{s}(\x_{*})) < P_{LS}(\mathcal{C}_{s}(\x_{*}))$, in generic cases when the entries of  
$\A \in \mathbb{R}^{m \times n}$ are i.i.d Gaussians $\mathcal{N}(0,\frac{1}{m})$, the recovered signals belong to parsimonious models and feasible sets are appropriately chosen.
In Appendix~\ref{app:conjecture} we present 
experiments that support our conjecture.

\textbf{Remark.} 
\tomt{
As typically done in the optimization literature (e.g., see \cite{beck2009fast,beck2017first,oymak2017sharp}), we have 
\tomt{mathematically} examined {\em upper bounds}  
on the convergence rates \tomt{of the optimization algorithm (PGD in our case)}. 
Since we wish to compare the \tomt{practical} convergence rates of PGD for LS and BP, 
a natural question is:  Should the bounds be tight in order to deduce 
\tomt{theoretically backed} 
conclusions on the relation of the {\em real} rates for LS and BP 
\tomt{(i.e., which one is faster)}?
Interestingly, when both objectives lead to a similar stationary point $\x_{*}$, 
it is {\em enough} to verify that $P_{LS}(\mathcal{C}_{s}(\x_{*}))$ is tight \tomt{in order} to conclude that the real rate for BP is better than for LS.} 
This follows from the fact that the real rate of BP is smaller (i.e., better) than $P_{BP}(\mathcal{C}_{s}(\x_{*}))$, and 
that 
$P_{BP}(\mathcal{C}_{s}(\x_{*})) \leq P_{LS}(\mathcal{C}_{s}(\x_{*}))$. 
Thus, 
tightness in $P_{LS}$ is important for this conclusion (and is indeed obtained in certain cases, as discussed above and empirically demonstrated in \cite{oymak2017sharp}), 
while 
``miss-tightness" in $P_{BP}$ only increases the gap between the {\em real} rates of LS and BP in favor of BP.

\subsection{Extrinsic Source of Different Convergence Rates}
\label{sec:convergence_ext}

Since using LS and BP objectives in \eqref{Eq_cost_func_general2} defines two different optimization problems, potentially, one may prefer to assign different values for the regularization parameter $R$ in each case. This is obviously translated to using feasible sets with different volume. 
Note that the obtained convergence rates depend on the feasible set through $\mathcal{D}_{s}(\x_{*})$ and $\mathcal{C}_{s}(\x_{*})$, and are therefore affected by the value of $R$. 
We refer to this effect on the convergence rate as ``extrinsic" because it originates in a modified prior rather than directly from the different BP and LS objectives.

For the LS objective, under the assumption of Gaussian $\A$, the work 
\cite{oymak2017sharp} has used the notion of Gaussian width to theoretically link 
the complexity of the signal prior, which translates to the feasible set in \eqref{Eq_cost_func_general2}, 
and the convergence rate of PGD.
Their result implies that {\em increasing} the size of the feasible set (due to a relaxed prior) is expected to {\em decrease} the convergence rate, i.e., {\em slow down} PGD.  
Therefore, it is expected that using $R_{BP}<R_{LS}$ would increase the gap between the convergence rates in favor of the BP term, beyond the effect of its intrinsic advantage described in Section \ref{sec:convergence_intrin}.
On the other hand, 
using $R_{BP}>R_{LS}$ may counteract the intrinsic advantage of BP.

\section{Convergence Analysis Beyond PGD}
\label{sec:convergence_beyond}

Many works on inverse problems use the penalized optimization problem \eqref{Eq_cost_func_general} rather than the constrained one \eqref{Eq_cost_func_general2}. 
Oftentimes \eqref{Eq_cost_func_general} is minimized using the proximal gradient method, which is given by
\begin{align}
\label{Eq_ista}
{\x}_{t+1} =  \mathrm{prox}_{\mu \beta s(\cdot)}({\x}_{t} - \mu \nabla  \ell({\x}_{t})),
\end{align}
where 
\begin{align}
\label{def_prox}
\mathrm{prox}_{s(\cdot)}({\z}) \triangleq \argmin{{\x}} \,\, \frac{1}{2} \| {\z} - {\x} \|_2^2 + s({\x})
\end{align}
is the proximal mapping $s(\cdot)$ at the point ${\z}$, which was introduced for convex functions in \cite{moreau1965proximite}.
Note that PGD with a convex feasible set is essentially the proximal gradient method for $s(\cdot)$ which is a convex indicator, 
and similarly to PGD,
setting the step-size $\mu$ to 1 over the Lipschitz constant of $\nabla  \ell(\cdot)$ ensures sublinear convergence of \eqref{Eq_ista}
in convex settings \cite{beck2009fast}.

Note that the proximal mapping of any convex $\beta s(\cdot)$ is non-expansive (see, e.g., \cite{beck2017first}), i.e., for all ${\z}_1, {\z}_2$
\begin{align}
\label{Eq_non_expansive}
\| \mathrm{prox}_{\beta s(\cdot)}({\z}_1) - \mathrm{prox}_{\beta s(\cdot)}({\z}_2) \|_2 \leq \| {\z}_1 - {\z}_2 \|_2. 
\end{align}
However,
this property is not enough to obtain an expression that allows to distinguish between the convergence rates of LS and BP (as done using \eqref{Eq_pgd_thm} for PGD),
because it does not express the effect of the prior on the null space of $\A$.

To obtain an expression that allows to compare the convergence rates of \eqref{Eq_ista} with $\ell_{LS}({\x})$ and $\ell_{BP}({\x})$, we make a relaxed contraction assumption.
Namely, we require that the proximal mapping of $\beta s(\cdot)$ is a contraction (only) in the null space of $\A$ (rather than in all $\mathbb{R}^n$).

\begin{condition}
\label{cond3}
Given the convex function $\beta s(\cdot)$ and the full row-rank matrix $\A$, there exists $0< \tomt{\delta_{\A,\beta s(\cdot)}} \leq 1$ such that for all ${\z}_1, {\z}_2$ 
\begin{align}
\label{Eq_contraction}
&\| \mathrm{prox}_{\beta s(\cdot)}({\z}_1) - \mathrm{prox}_{\beta s(\cdot)}({\z}_2) \|_2 \leq  
\left \| \left (\P_A + (1-\tomt{\delta_{\A,\beta s(\cdot)}})\Q_A \right ) ( {\z}_1 - {\z}_2 ) \right \|_2. 
\end{align}
\end{condition}

The quantity 
$\delta_{\A,\beta s(\cdot)}$ in Condition~\ref{cond3} reflects the restriction that the prior $\beta s(\cdot)$ imposes on the null space of $\A$. For the restrictive prior \eqref{Eq_oracle_prior}, given in the warm-up Section~\ref{sec:toy}, it is easy to see that $\mathrm{prox}_{\beta s_{oracle}(\cdot)}({\z}) = \P_{A}{\z} + \Q_A\x_{gt}$, which implies $\delta_{\A,\beta s_{oracle}(\cdot)}=1$.
On the other hand, the general property in \eqref{Eq_non_expansive} is obtained for $\delta_{\A,\beta s(\cdot)}=0$ (because $\P_{A}+\Q_{A} = \I_n$). 
Condition~\ref{cond3} is weaker than requiring that 
$\mathrm{prox}_{\beta s(\cdot)}(\cdot)$
is a contraction in all $\mathbb{R}^n$. 
Thus, 
it holds for priors that satisfy the latter \cite{tikhonov1963solution, teodoro2018convergent}. See Appendix~\ref{app:Contraction} for more details on this condition.

The following theorem shows that if Condition~\ref{cond3} holds \tomt{(namely, the prior imposes restrictions on the null space of $\A$)}, then the iterates \eqref{Eq_ista} with step-size of $\| \nabla^2 \ell \|^{-1}$ exhibit a {\em linear} convergence under 
conditions that are satisfied by both $\ell_{LS}({\x})$ and $\ell_{BP}({\x})$. 
The proof appears in Appendix~\ref{app:proof_ista}.

\begin{theorem}
\label{theorem_ista}
Let $s : \mathbb{R}^n \rightarrow \mathbb{R}$ be a convex function and let $\ell : \mathbb{R}^n \rightarrow \mathbb{R}$ be a twice differentiable convex function that satisfies $\nabla\ell(\cdot) \in \mathrm{range}(\A^T)$ for a given full row-rank matrix $\A$.
Denote by $\tilde{\sigma}_{max}$ the largest eigenvalue of $\nabla^2 \ell$ and by $\tilde{\sigma}_{min}$ the smallest {\em non-zero} eigenvalue of $\nabla^2 \ell$.
Then, if Condition~\ref{cond3} holds for $\mu\beta s(\cdot)$ and $\A$, we have that the sequence $\{{\x}_t\}$ obtained by \eqref{Eq_ista} with $\mu=1/\tilde{\sigma}_{max}$ obeys
\begin{align}
\label{Eq_ista_thm}
&\|{\x}_{t+1}-\x_*\|_2 \leq  
\mathrm{max} \left \{ 1-\frac{\tilde{\sigma}_{min}}{\tilde{\sigma}_{max}} , 1 - \delta_{\A,\frac{\beta}{\tilde{\sigma}_{max}}s(\cdot)}  \right \} \|{\x}_{t}-\x_*\|_2,
\end{align}
where $\x_*$ is a minimizer of \eqref{Eq_cost_func_general}.
\end{theorem}

For LS we have that $\nabla \ell_{LS}({\x}) = - \A^T ( \y - \A {\x} )$ and $\nabla^2 \ell_{LS}({\x}) = \A^T \A$. Therefore, $\tilde{\sigma}_{max} = \sigma_{max}(\A\A^T)$ and $\tilde{\sigma}_{min} = \sigma_{min}(\A\A^T)$, and Theorem~\ref{theorem_ista} implies
\begin{equation}
\label{Eq_ista_thm_ls}
\frac{\|{\x}_{t+1}^{LS}-\x_*^{LS}\|_2}{\|{\x}_{t}^{LS}-\x_*^{LS}\|_2} \leq \mathrm{max} \left \{ 1- {\displaystyle   \frac{\sigma_{min}(\A\A^T)}{\sigma_{max}(\A\A^T)}} , 1 - \tilde{\delta}_{LS}  \right \},
\end{equation}
where $\tilde{\delta}_{LS} \triangleq \delta_{\A,\frac{\beta_{LS}}{\sigma_{max}(\A\A^T)}s(\cdot)}$.

For BP we have that $\nabla \ell_{BP}({\x}) = - \A^\dagger ( \y - \A {\x} )$ and $\nabla^2 \ell_{BP}({\x}) = \A^\dagger \A$. Therefore, $\tilde{\sigma}_{max} = 1$ and $\tilde{\sigma}_{min} = 1$, and Theorem~\ref{theorem_ista} implies
\begin{equation}
\label{Eq_ista_thm_bp}
\frac{\|{\x}_{t+1}^{BP}-\x_*^{BP}\|_2}{\|{\x}_{t}^{BP}-\x_*^{BP}\|_2} \leq  1 - \tilde{\delta}_{BP},
\end{equation}
where $\tilde{\delta}_{BP} \triangleq \delta_{\A,\beta_{BP}s(\cdot)}$.

Comparing \eqref{Eq_ista_thm_ls} and \eqref{Eq_ista_thm_bp}, it can be seen that if Condition~\ref{cond3} holds then there is an advantage for the BP term over the LS term, which is due to a better ``restricted condition number" of the Hessian of $\ell_{BP}$ in the row space of $\A$. 
Specifically, 
note that if $\tilde{\delta}_{LS}<\tilde{\delta}_{BP}$ then the bound on the rate of BP is better, {\em regardless} of $\frac{\sigma_{min}(\A\A^T)}{\sigma_{max}(\A\A^T)}$.
Alternatively, the results hint 
that a worse condition number of $\A\A^T$ is expected to correlate with a larger difference between the convergence rates of LS and BP {\em in favor of} BP.
Since PGD with a convex feasible set is a special case of the proximal gradient method,  
\tomt{
these results apply also to PGD. 
Indeed, such a behavior is demonstrated in our experiments for compressed sensing tasks with $\ell_1$-norm prior (see Fig.~\ref{fig:CS_results_different_ratios} in the sequel), despite the fact that Condition~\ref{cond3} (which is difficult to be verified in general) may not hold in that case. This shows the practical implication of our theoretical result beyond the strict settings required to prove the theorem.}

In this paper we mainly focus on direct PGD results (rather than on those obtained for general proximal methods) for two reasons.
Firstly, they do not require a contractive assumption. 
Secondly, identifying an ``intrinsic factor'' for different convergence rates is easier for PGD both in 
the experiments 
(as discussed on Fig.~\ref{fig:CS_results_fista} in the sequel) 
and the analysis (the dependence of $\tilde{\delta}_{LS},\tilde{\delta}_{BP}$ on $\beta_{LS},\beta_{BP}$ is not explicit and 
cannot be bypassed by assuming $\x_{*}^{LS} \approx \x_{*}^{BP}$, as we have done in Section~\ref{sec:convergence_intrin} 
to identify the inherent advantage of $P_{BP}$ over $P_{LS}$ for PGD).

\section{Experiments}
\label{sec:exp}

In this section, we provide numerical experiments that corroborate our 
analysis for both convex ($\ell_1$-norm) and non-convex  
(DCGAN \cite{radford2015unsupervised}) priors.
\tomt{In Section~\ref{sec:exp_l1}, we consider the $\ell_1$-norm prior and} 
examine the performance of PGD with LS and BP objectives for compressed sensing (CS).  
It is demonstrated that both objectives
prefer (i.e., provide better PSNR\footnote{The PSNR of $\hat{\x}$ with respect to the reference image $\x_{gt} \in [0,255]^n$ is defined as $10\mathrm{log}_{10}\left ( \frac{255^2}{\frac{1}{n}\|\hat{\x}-\x_{gt}\|_2^2} \right )$.} for) a similar value of $R$ --- a case in which the faster convergence for BP is dictated by its ``intrinsic" advantage, rather than by an ``extrinsic" source. 
We also
examine
an accelerated proximal gradient method (FISTA \cite{beck2009fast}) applied on \eqref{Eq_cost_func_general} with LS and BP fidelity terms, and suggest an explanation for the observed behavior using the ``extrinsic" and ``intrinsic" sources.
\tomt{In Section~\ref{sec:exp_l1_linear}, still considering the $\ell_1$-norm prior, we run a few controlled experiments (where the conditions of our theorem hold) that demonstrate the linear convergence of PGD more clearly.} 
\tomt{Finally, in Section~\ref{sec:exp_dcgan}, we turn to consider the DCGAN prior. We}
examine the performance of PGD for compressed sensing (CS) and super-resolution (SR) tasks, and show again the inherent advantage of the BP objective.

\subsection{$\ell_1$-Norm Prior}
\label{sec:exp_l1}

\tomt{We consider the CS task,
where a signal in $\mathbb{R}^n$ needs to be recovered from $m$ compressed and noisy measurements $(m<n)$.
Specifically, we consider a typical setting,} 
where the measurement matrix is Gaussian (with i.i.d.~entries drawn from $\mathcal{N}(0,1/m)$), the compression ratio is $m/n=0.5$, and the signal-to-noise ratio (SNR) is 20dB (with white Gaussian noise).
We use four standard test images: {\em cameraman}, {\em house}, {\em peppers}, and {\em Lena}, in their $128 \times 128$ versions (so $n=128^2$). 
To apply sparsity-based recovery, we represent the images in the Haar wavelet basis, 
i.e., $\A$ is the multiplication of the measurement matrix with the Haar basis.

For the reconstruction, we use the feasible set $\mathcal{K} = \left \{ {\x} \in \mathbb{R}^n : \|{\x}\|_1 \leq R \right \}$, where $\|\cdot\|_1$ is the $\ell_1$-norm, and project on it using the fast algorithm from \cite{duchi2008efficient}. Starting from ${\x}_0=0$, we apply 1000 iterations of PGD on the BP and LS objectives with the typical step-size of 1 over the spectral norm of the objective's Hessian.
We compute $\A^\dagger$ in advance. 
Thus, PGD has {\em similar per-iteration computational cost} for both objectives and the overall complexity is dictated by the number of iterations.

\begin{figure}[t]
  \centering
  \begin{subfigure}[b]{0.5\linewidth}
    \centering\includegraphics[width=200pt]{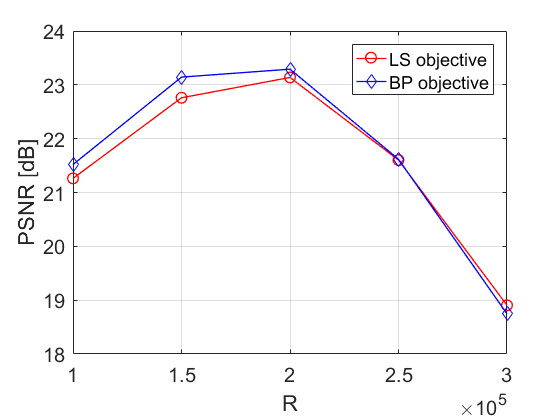}
    \caption{\label{fig:CS_results_psnr_vs_R}}
  \end{subfigure}%
  \begin{subfigure}[b]{0.5\linewidth}
    \centering\includegraphics[width=200pt]{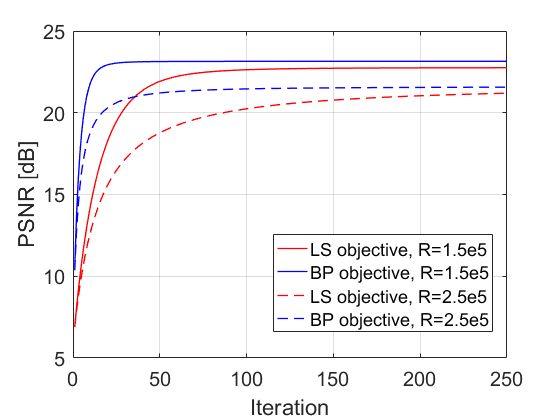}
    \caption{\label{fig:CS_results_psnr_vs_iter}}
  \end{subfigure}%
  \caption{\tomt{PSNR results (averaged over 4 test images) of PGD with $\ell_1$ prior for the compressed sensing task with $m/n=0.5$ Gaussian measurements and SNR of 20dB.
  (\subref{fig:CS_results_psnr_vs_R}): PSNR vs.~regularization parameter $R$ (after 1K iterations). Observe that both LS and BP prefer similar values of $R$.
  (\subref{fig:CS_results_psnr_vs_iter}): PSNR vs.~iteration number (for $R=1.5\mathrm{e}5$ and $R=2.5\mathrm{e}5$). Observe the inherent convergence advantage of BP when both LS and BP use the same value for $R$.}}
\label{fig:CS_results}

\vspace{1mm}

  \centering
    \centering\includegraphics[width=200pt]{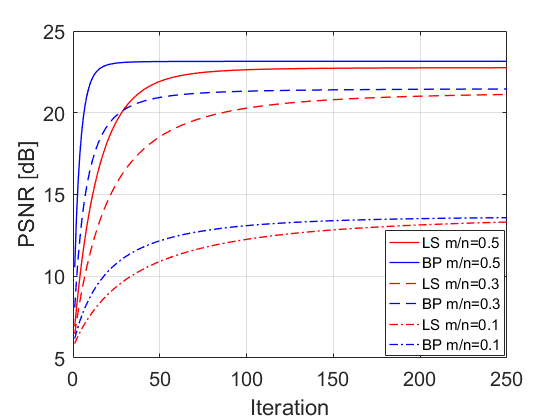}
  \caption{\tomt{PSNR results (averaged over 4 test images) vs.~the iteration number of PGD with $\ell_1$ prior and $R=1.5\mathrm{e}5$ for the compressed sensing task with different $m/n$ ratios  
  and SNR of 20dB.
  Note that 
$\frac{\sigma_{min}(\A\A^T)}{\sigma_{max}(\A\A^T)}$ equals 0.0296, 0.0862 and 0.2721 for $m/n$ ratios of 0.5, 0.3 and 0.1, respectively.
  The convergence advantage of BP over LS is larger when the condition number of $\A\A^T$ is worse.}}    
\label{fig:CS_results_different_ratios}
\vspace{-5mm}
\end{figure}

Fig.~\ref{fig:CS_results_psnr_vs_R} shows the PSNR of the reconstructions, averaged over all images, for different values of the regularization parameter $R$.
Fig.~\ref{fig:CS_results_psnr_vs_iter} shows the average PSNR as a function of the iteration number, for $R=1.5\mathrm{e}5$ and $R=2.5\mathrm{e}5$. Note that $R=2.5\mathrm{e}5$ yields less accurate results despite being the average $\ell_1$-norm of the four ``ground truth" test images (in Haar basis representation). 
\tomt{Importantly, from
Fig.~\ref{fig:CS_results_psnr_vs_iter} we see that when PGD is applied on BP and LS objectives with the same value of $R$, indeed BP is faster, which demonstrates its ``intrinsic" advantage.} 
Also, when $R$ is increased, the convergence of PGD for both objectives becomes slower due to this ``extrinsic" modification.
Note, though, that Fig.~\ref{fig:CS_results_psnr_vs_R} implies that both objectives prefer a similar value of $R$. Therefore, when $R$ is (uniformly) tuned for best PSNR of each method, it is expected that the intrinsic advantage of BP over LS is the reason for its faster PGD convergence.

\begin{figure}[t]
  \centering
  \begin{subfigure}[b]{0.5\linewidth}
    \centering\includegraphics[width=200pt]{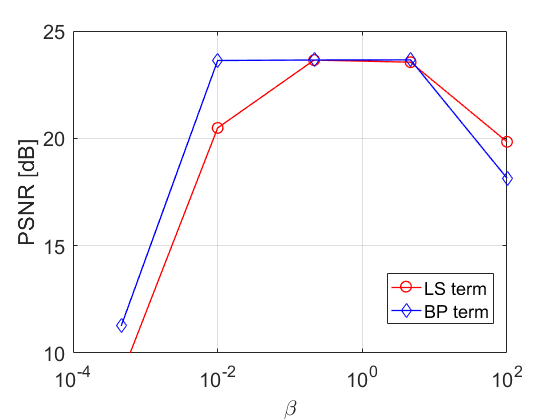}
    \caption{\label{fig:CS_results_psnr_vs_beta_fista}}
  \end{subfigure}%
  \begin{subfigure}[b]{0.5\linewidth}
    \centering\includegraphics[width=200pt]{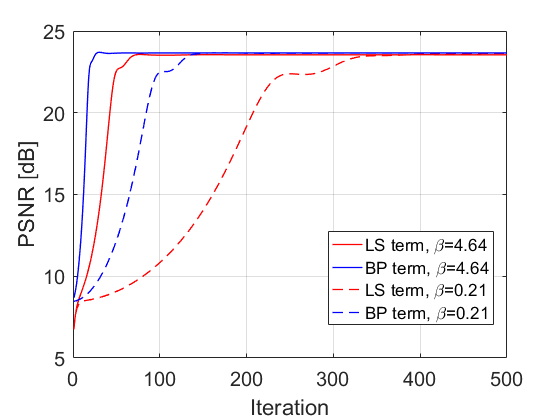}
    \caption{\label{fig:CS_results_psnr_vs_iter_fista}}
  \end{subfigure}%
  \caption{\tomt{PSNR results (averaged over 4 test images) of FISTA with $\ell_1$ prior for the compressed sensing task with $m/n=0.5$ Gaussian measurements and SNR of 20dB.
  (\subref{fig:CS_results_psnr_vs_beta_fista}): PSNR vs.~regularization parameter $\beta$ (after 1K iterations). Observe that similar PSNR values are obtained by different values of $\beta$.
  (\subref{fig:CS_results_psnr_vs_iter_fista}): PSNR vs.~iteration number (for $\beta=0.21, 4.64$). Observe the strong effect of $\beta$ on the convergence rate.}}  
\label{fig:CS_results_fista}

\vspace{1mm}

  \centering
    \centering\includegraphics[width=200pt]{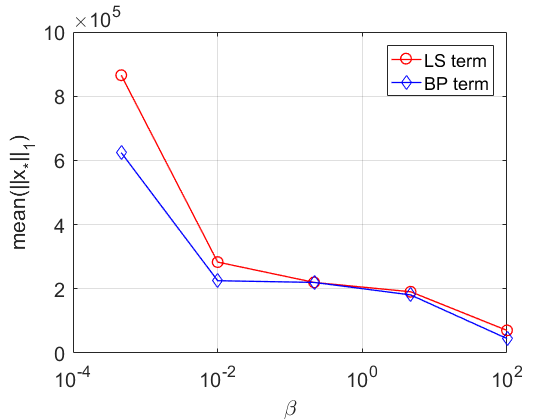}
  \caption{\tomt{Average $\|\x_*\|_1$ of the reconstructions of FISTA from Fig.~\ref{fig:CS_results_psnr_vs_beta_fista} vs.~the regularization parameter $\beta$. Note that similar values of $\|\x_*\|_1$ are obtained by different values of $\beta$.
  }}
\label{fig:CS_results_fista_avgL1}
\vspace{-5mm}
\end{figure}

Next, we examine the convergence rates of PGD with LS and BP objectives for different compression ratios $m/n$ and fixed $R=1.5\mathrm{e}5$ (still, with SNR of 20dB). 
Fig.~\ref{fig:CS_results_different_ratios} shows the average PSNR 
vs.~iteration number, for $m/n=0.5$, $m/n=0.3$ and $m/n=0.1$. 
Note that in these experiments the ratio $\frac{\sigma_{min}(\A\A^T)}{\sigma_{max}(\A\A^T)}$ equals 0.0296, 0.0862 and 0.2721 for $m/n$ ratios of 0.5, 0.3 and 0.1, respectively.
Observing the convergence rates of the different curves in this figure, it is easy to see that the advantage of the rate of BP over the rate of LS increases when the ratio $m/n$ increases, or alternatively when the ratio $\frac{\sigma_{min}(\A\A^T)}{\sigma_{max}(\A\A^T)}$ decreases.

This empirical behavior is inline with the analysis in both Section~\ref{sec:convergence_intrin} and Section~\ref{sec:convergence_beyond}.
Section~\ref{sec:convergence_intrin} characterizes the ratio between the convergence rates of LS and BP by the ratio of the terms $P_{LS}$ and $P_{BP}$. 
An approximation of ${P}_{BP}/{P}_{LS}$ is provided  
in Appendix~\ref{app:conjecture}  
for a similar CS setting. 
It is shown there (in Fig.~\ref{fig:approx_ratio_vs_m}) that the ratio ${P}_{BP}/{P}_{LS}$ decreases (i.e., the advantage of BP increases) when the ratio $m/n$ increases, which indeed agrees with Fig.~\ref{fig:CS_results_different_ratios}.
Section~\ref{sec:convergence_beyond} considers the proximal gradient method, which subsumes PGD, and the results there in \eqref{Eq_ista_thm_ls} and \eqref{Eq_ista_thm_bp} suggest that the convergence rate of BP can be less affected than the  
one of LS by a bad condition number of $\A\A^T$, i.e., low values of $\frac{\sigma_{min}(\A\A^T)}{\sigma_{max}(\A\A^T)}$. Again, this agrees with the results in Fig.~\ref{fig:CS_results_different_ratios}.

We turn now to recover the images by minimizing  
\eqref{Eq_cost_func_general} using 1000 iterations of FISTA \cite{beck2009fast} with LS and BP fidelity terms. 
We consider again the case of $m/n=0.5$. 
Figs.~\ref{fig:CS_results_psnr_vs_beta_fista} and \ref{fig:CS_results_psnr_vs_iter_fista} show the average PSNR vs.~$\beta$ and vs.~iteration number, respectively.
Fig.~\ref{fig:CS_results_fista_avgL1} presents the average $\|\x_*\|_1$ of the recoveries vs.~$\beta$.
Note that the best PSNRs for BP and LS are received by values of $\beta$ for which $\|\x_*\|_1$ is very similar for both terms, i.e., the {\em equivalent} constrained LS and BP formulations have very similar $R$ (as observed for PGD).

However, disentangling the factors for different convergence rates of LS and BP, where for each of them the regularization parameter is (uniformly) tuned for best PSNR, is more complicated for proximal methods than for PGD. 
To see this, note that 
in Figs.~\ref{fig:CS_results_psnr_vs_beta_fista} and \ref{fig:CS_results_fista_avgL1} for each fidelity term similar values of PSNR and $\|\x_*\|_1$ can be obtained for different values of $\beta$.
Yet, as shown in Figs.~\ref{fig:CS_results_psnr_vs_iter_fista}, different values of $\beta$ significantly change the convergence rate of FISTA for the same fidelity term.  
Thus, contrary to our conclusion for PGD, 
here when $\beta$ is uniformly tuned for best PSNR of each fidelity 
term (as in \cite{tirer2019back}), an ``extrinsic source" ($\beta$ setting) can affect the convergence rate as well.

\begin{figure}[t]
  \centering
    \centering\includegraphics[width=200pt]{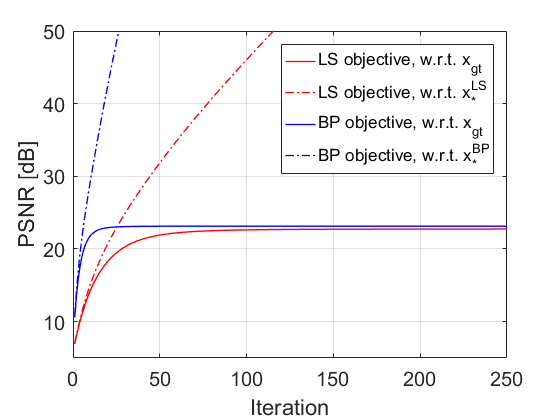}
  \caption{\tomt{PSNR results (w.r.t.~both ground truth and the final stationary point, averaged over 4 test images) vs.~the iteration number of PGD with $\ell_1$ prior and $R=1.5\mathrm{e}5$ for the compressed sensing task with $m/n=0.5$ Gaussian measurements and SNR of 20dB.
  Observe that measuring the PSNR w.r.t.~the final stationary point shows that the convergence is linear (but does not reflect accuracy).}}  
\label{fig:CS_results_psnr_vs_iter_stat_point}


\vspace{1mm}

  \centering
  \begin{subfigure}[b]{0.5\linewidth}
    \centering\includegraphics[width=200pt]{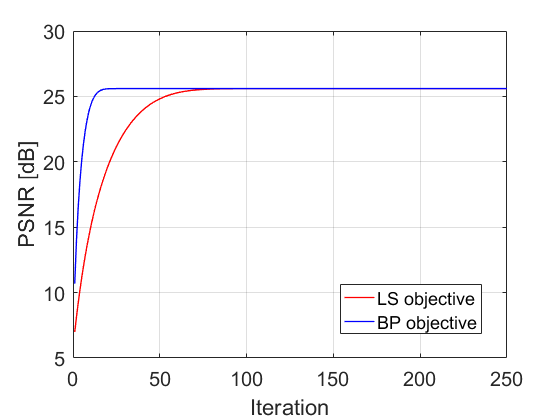}
    \caption{\label{fig:CS_results_psnr_vs_iter_snr20_sparse}}
  \end{subfigure}%
  \begin{subfigure}[b]{0.5\linewidth}
    \centering\includegraphics[width=200pt]{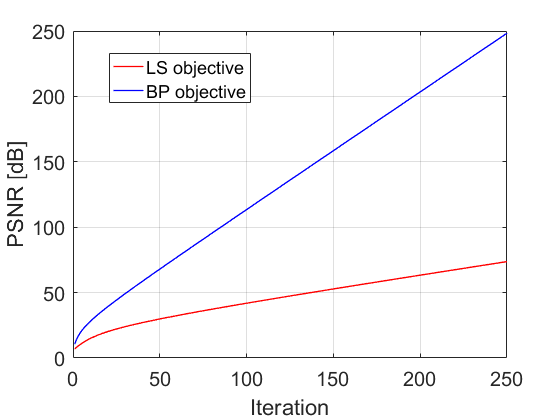}
    \caption{\label{fig:CS_results_psnr_vs_iter_snrInf_sparse}}
  \end{subfigure}%
  \caption{\tomt{PSNR results (averaged over 4 test images) vs.~the iteration number of PGD with $\ell_1$ prior and $R=\|\x_{gt}\|_1$ (per image) for the compressed sensing task with $m/n=0.5$ Gaussian measurements and sparsified test images. 
  (\subref{fig:CS_results_psnr_vs_iter_snr20_sparse}): SNR = 20dB.
  (\subref{fig:CS_results_psnr_vs_iter_snrInf_sparse}): SNR = $\infty$.
  Observe the faster convergence of BP, especially in the noiseless scenario, where linear convergence to the ground truth is guaranteed.}} 
\label{fig:CS_results_controlled}

\vspace{-5mm}
\end{figure}

\subsection{Linear Convergence for $\ell_1$-Norm Prior}
\label{sec:exp_l1_linear}

\tomt{
In the experiments above, we aimed to show that the insights about the convergence advantage when using the BP term rather than the LS term are reflected in {\em practical applications}. Therefore, we used {\em natural images} and regularization parameter setting that is {\em uniform} across all test images. The quality of the reconstructed image at each iteration was measured by its PSNR with respect to (w.r.t.) the ground truth image $\x_{gt}$, which is the most common quality assessment measure that is used by practitioners. While the inherent convergence advantage of BP and its dependence on the condition number of $\A\A^T$ are already observed in these experiments, the ``PSNR (w.r.t. $\x_{gt}$) vs.~iteration" curves themselves may not display the linear convergence that is 
suggested by our theory. Our goal in this subsection is to close this gap.} 

\tomt{
To this end, we first 
examine the PSNR w.r.t.~the final stationary point of the PGD, i.e., $\x_{*}$ (obtained by a preceding application of the algorithm), rather than w.r.t.~$\x_{gt}$. Such an example appears in Fig.~\ref{fig:CS_results_psnr_vs_iter_stat_point}, where we consider Gaussian CS with $m/n = 0.5$ and SNR of 20dB, and present PSNR curves (averaged over the previous 4 test images), w.r.t.~both $\x_{gt}$ and $\x_{*}$, for PGD with $\ell_1$ prior, $R=1.5\mathrm{e}5$ and $\x_0=0$. The curves of PSNR w.r.t.~$\x_{gt}$ are the same as those that appear in Fig.~\ref{fig:CS_results_psnr_vs_iter}. The slopes of the two types of PSNRs are similar at early iterations, displaying the convergence advantage of the BP objective. The curves for the PSNR that is measured w.r.t.~$\x_{*}$ exhibit a more consistent linear shape. Yet, they mask the estimation accuracy (i.e., the distance from the true image, $\x_{gt}$), which is probably the most important property of a reconstruction method.} 

\tomt{
To further verify the linear convergence theory without compromising on estimation accuracy information, we turn to perform {\em controlled experiments}. We modify each ground truth image $\x_{gt}$ to make it a purely sparse signal: we keep only its 800 dominant Haar wavelet coefficients, which is slightly less than $\frac{m}{\mathrm{log}(n)}=\frac{0.5 \times 128^2}{\mathrm{log}(128^2)} \approx 844$. For this sparsity level, $\ell_1$ prior and number of Gaussian CS measurements, existing theory (e.g., in \cite{chandrasekaran2012convex}) ensures that $P_{LS}(\mathcal{C}_{s}(\x_{gt}))<1$, and thus by Proposition~\ref{prop1} it is also guaranteed that $P_{BP}(\mathcal{C}_{s}(\x_{gt}))<1$. For {\em each} test image we use $R=\|\x_{gt}\|_1$ (rather than a uniform parameter setting), which allows invoking Theorem~\ref{theorem1} with $\x_* = \x_{gt}$ (instead of with general stationary points) for which perfect reconstruction is ensured in the noiseless case. Under these controlled settings, we perform CS experiments with $m/n = 0.5$ and SNR of 20dB as well as without noise (SNR=$\infty$).} 

\tomt{
Figs.~\ref{fig:CS_results_psnr_vs_iter_snr20_sparse} and \ref{fig:CS_results_psnr_vs_iter_snrInf_sparse} show the PSNR w.r.t.~$\x_{gt}$ (averaged over the 4 test images) of the PGD reconstruction vs.~the iteration number for the two SNR levels. 
For SNR of 20dB both methods reach a plateau, associated with the noise term in \eqref{Eq_pgd_thm}. The curves are nearly linear at early iterations and have similar trends as those obtained in Fig.~\ref{fig:CS_results_psnr_vs_iter} for the uncontrolled experiment. In the noiseless case (SNR=$\infty$), the noise term is eliminated, and the methods can reach perfect accuracy with exact linear rates, in agreement with \eqref{Eq_pgd_thm_noiseless}. Both figures clearly 
demonstrate our theory and 
show the convergence advantage of using the BP objective.}

\subsection{DCGAN Prior}
\label{sec:exp_dcgan}

The recent advances in learning (deep) generative models have 
led to using them 
as priors in imaging inverse problems (see, e.g., \cite{bora2017compressed, shah2018solving, shady2019image}). 
\tomt{
In order to generate new samples that are similar to the training samples (e.g., samples of human faces), popular generative models, such as VAEs \cite{kingma2013auto} and GANs \cite{goodfellow2014generative}, learn a nonlinear transformation $\mathcal{G}(\cdot)$, typically referred to as the ``generator", that maps a random Gaussian {\em noise} vector of small dimension $\z \in \mathbb{R}^d$ to the signal space in $\mathbb{R}^n$ ($d \ll n$). 
The common approach to use such models as priors 
is to search for a reconstruction of $\x_{gt}$ only in the range of a pre-trained generator,}
i.e., in 
$\mathcal{K}_{\mathcal{G}} = \left \{ {\x} \in \mathbb{R}^n : \exists {\z} \in \mathbb{R}^d \,\,\, \mathrm{s.t.} \,\,\,  {\x}=\mathcal{G}({\z}) \right \}$. 
Note that the proposed PGD theory, which assumes in \eqref{Eq_set_Kr} that 
$\mathcal{K} = \left \{ {\x} \in \mathbb{R}^n : s({\x}) \leq R \right \}$, covers the above feasible set for $R=0$ and the 
non-convex 
$s({\x})=
\Big \{
    \begin{array}{lr}
      0,& {\x} \in \mathcal{K}_{\mathcal{G}} \\
      +\infty,& otherwise
    \end{array}$.

In the next experiments we use $\mathcal{G}(\z)$ that we obtained by training a DCGAN \cite{radford2015unsupervised} on the first 200,000 images (out of 202,599) of CelebA dataset. 
We use the $64 \times 64$ version of the images (so $n=64^2$) and a training procedure similar to \cite{radford2015unsupervised, bora2017compressed}. 
We start with the CS scenario from previous section, where $m/n=0.5$, the entries of the measurement matrix are i.i.d.~drawn from $\mathcal{N}(0,1/m)$, and the SNR is 20dB.
The last 10 images in CelebA are used as test images.

The recovery using each of the LS and BP objectives is based on 50 iterations of PGD with the typical step-size 
and initialization of ${\x}_0=\A^\dagger\y$.
As the projection $\mathcal{P}_{\mathcal{K}}({\x})$ we use $\mathcal{G}(\hat{\z})$, where $\hat{\z}$ is obtained by minimizing $\| {\x}- \mathcal{G}({\z}) \|_2^2$ with respect to ${\z}$. This inner  
minimization problem is carried out by 1000 iterations of ADAM \cite{kingma2014adam} with 
LR of 0.1 and multiple initializations. 
The value of ${\z}$ that gives the lowest $\| {\x}- \mathcal{G}({\z}) \|_2^2$ is chosen as $\hat{\z}$.
For the projection in the first PGD iteration we use the same 10 random initializations of ${\z}$ for both LS and BP. 
Projections in other PGD iterations use warm start from the preceding iteration. 
For both LS and BP 
the computational cost of a PGD iteration is similar, 
since the matrices $\A^T$ and $\A^\dagger$ are computed in advance. 
Moreover, now the per-iteration complexity is dominated by the  
projection. 
Thus, again, the overall complexity 
is dictated by the number of iterations.

\begin{figure}[t]
  \centering
  \begin{subfigure}[b]{0.5\linewidth}
    \centering\includegraphics[width=200pt]{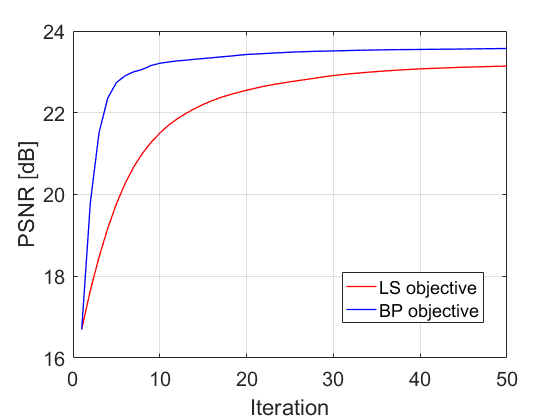}
    \caption{\label{fig:CS_results_psnr_vs_iter_dcgan}}
  \end{subfigure}%
    \begin{subfigure}[b]{0.5\linewidth}
    \centering\includegraphics[width=200pt]{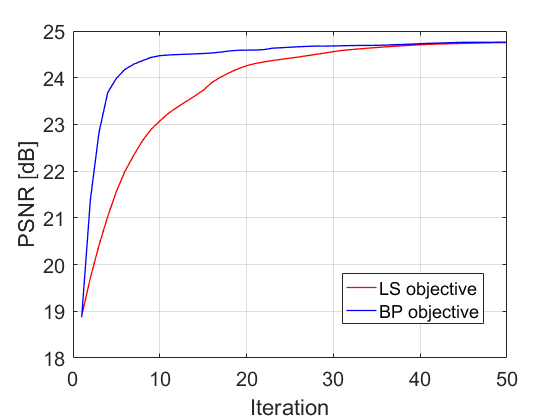}
    \caption{\label{fig:CS_results_psnr_vs_iter_dcgan_202592}}
  \end{subfigure}%
  \caption{\tomt{PSNR results vs.~the iteration number of PGD with DCGAN prior for the compressed sensing task with $m/n=0.5$ Gaussian measurements and SNR of 20dB.
  (\subref{fig:CS_results_psnr_vs_iter_dcgan}): PSNR results averaged over 10 CelebA test images.
  (\subref{fig:CS_results_psnr_vs_iter_dcgan_202592}): PSNR results for image 202592 in CelebA.
  In (\subref{fig:CS_results_psnr_vs_iter_dcgan_202592}), both LS and BP reach similar recoveries. The faster convergence of BP there implies its inherent advantage.
  The similarity between the rates in  
  (\subref{fig:CS_results_psnr_vs_iter_dcgan}) to those in  (\subref{fig:CS_results_psnr_vs_iter_dcgan_202592}) 
  hints that the inherent convergence advantage of BP is essential also in the 
  experiments
  where the recoveries are not similar (see details in the text).}}
\label{fig:CS_results_dcgan}

\vspace{1mm}

  \centering
    \centering\includegraphics[width=200pt]{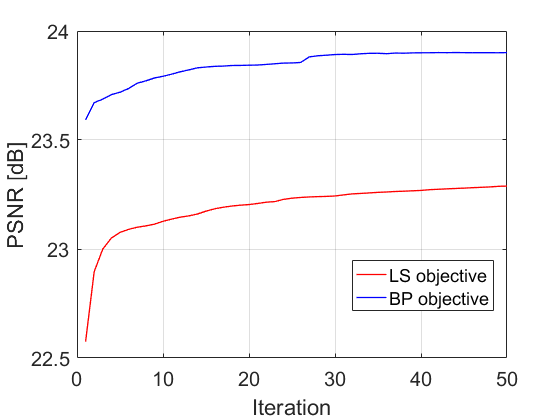}
  \caption{\tomt{PSNR results (averaged over 10 test images) vs.~the iteration number of PGD with DCGAN prior for the super-resolution task with Gaussian kernel and scale factor of 3.
  Note that the difference in the convergence rates is modest compared to the compressed sensing experiments. Yet, BP is still faster and yields more accurate results.}} 
\label{fig:SR_results_dcgan}
\vspace{-1em}
\end{figure}

Quantitative and visual CS recovery results appear in Appendix~\ref{app:visual}. 
In Fig.~\ref{fig:CS_results_psnr_vs_iter_dcgan} here, we show 
the average PSNR as a function of the iteration number.
Again, it is clear that BP objective requires significantly fewer iterations.
Since the DCGAN prior does not require a regularization parameter, the discussed ``extrinsic" source of faster convergence is not relevant. However, recall that DCGAN prior is (highly) non-convex, contrary to the $\ell_1$-norm prior. Therefore, $\x_{*}^{LS}$ and $\x_{*}^{BP}$, the PGD stationary points for LS and BP objectives, may be extremely different, and similarly, their two associated cones $\mathcal{C}_{s}(\x_{*}^{LS})$ and $\mathcal{C}_{s}(\x_{*}^{BP})$ may have very different geometries. 
This fact is another source for different convergence rates.

As an attempt to (approximately) isolate the effect of the intrinsic source on the convergence rates, we present in Fig.~\ref{fig:CS_results_psnr_vs_iter_dcgan_202592} the PSNR vs.~iteration number only for image 202592 in CelebA, where the recoveries using LS and BP objectives are relatively similar (see Fig.~\ref{fig:CS_dcgan2} in Appendix~\ref{app:visual}).
The similarity between the convergence rates in Figs.~\ref{fig:CS_results_psnr_vs_iter_dcgan} and \ref{fig:CS_results_psnr_vs_iter_dcgan_202592} hints that the inherent advantage of BP plays an essential role in its faster PGD convergence also for the other images in the examined scenario, 
where the recoveries are not similar.

Our final experiment considers 
a different observation model---the super-resolution (SR) 
task, where $\A$ composed 
of anti-aliasing filtering followed by down-sampling.
We use the widely examined scenario of scale factor 3  
and Gaussian filter of size $7 \times 7$ and standard deviation 1.6. For the reconstruction, we use PGD with DCGAN prior, initialized with bicubic upsampling of $\y$. 
Other configurations remain as before.

Fig.~\ref{fig:SR_results_dcgan} shows the average PSNR vs.~iteration number 
(more results appear in Appendix~\ref{app:visual}). 
Once again, the convergence of PGD for the BP objective is faster. However, this time the difference in the convergence rates is modest. Since in this SR experiment we have obtained significantly different recoveries for the LS and BP objectives 
(BP consistently yields higher PSNR, which can be explained by the analysis in \cite{tirer2019back}), 
we cannot try to isolate the effect of the intrinsic source, as done above. 
Yet, the results in this paper suggest that 
in this SR scenario 
the prior imposes a weaker restriction on the null space of $\A$ than in the CS scenario above.
In the analysis of Section~\ref{sec:convergence_intrin} this is translated to a smaller gap between $P_{LS}(\mathcal{C}_{s}(\x_{*}))$ and $P_{BP}(\mathcal{C}_{s}(\x_{*}))$,
and in the analysis of Section~\ref{sec:convergence_beyond} this is translated to a smaller contraction in Condition~\ref{cond3}.

\section{Conclusion}
\label{sec:conc}

In this paper we compared the convergence rate of PGD applied on LS and BP objectives, and identified an intrinsic source of a faster convergence for BP. 
Numerical experiments supported our theoretical findings for both convex ($\ell_1$-norm) and non-convex (pre-trained DCGAN) priors.
For the $\ell_1$-norm prior, we also provided numerical experiments that connected the PGD analysis with the behavior 
observed for proximal methods.  
A study of the latter has further highlighted 
BP's potential advantage when $\A\A^T$ is badly conditioned. 

\tomt{
In the constrained optimization problem that we studied (i.e., the problem in \eqref{Eq_cost_func_general2}) the objective is the data fidelity term, $\ell({\x})$, and the constraint is imposed on the prior term, $s({\x})$.
An interesting direction for future research is to compare the effects of the LS and BP terms on a different constrained form, which is also used in practice \cite{afonso2010augmented}, where the objective is the prior term and the constraint is on the fidelity term.
Our PGD analysis does not cover this form, because it requires the objective to be continuously differentiable (which is not obeyed by many popular priors) and builds on existing results related to cones induced by sublevel sets of some prior functions.}

\appendix

\section{Proofs for Section~\ref{sec:convergence}}
\label{app:proofs}

\subsection{Proof of Theorem~\ref{theorem1}}
\label{app:pgd}

In this section we prove Theorem~\ref{theorem1}.
To this end we adopt the following three lemmas from \cite{oymak2017sharp} (numbered there as Lemmas 16--18).

\begin{lemma}
\label{lemma:cone_norm}
Let $\mathcal{C} \subset \mathbb{R}^n$ be a closed cone and $\bm{v} \in \mathbb{R}^n$. Then
\begin{equation}
\label{Eq_cone_norm}
\| \mathcal{P}_{\mathcal{C}} (\bm{v}) \|_2 = \supr{ \bm{u} \in \mathcal{C} \cap \mathbb{B}^n } \bm{u}^T \bm{v}.
\end{equation}
\end{lemma}

\begin{lemma}
\label{lemma:proj_shift}
Let $\mathcal{K} \subset \mathbb{R}^n$ be a closed set and $\bm{u},\bm{v} \in \mathbb{R}^n$. The projection onto $\mathcal{K}$ obeys
\begin{equation}
\label{Eq_proj_shift}
\mathcal{P}_{\mathcal{K}} (\bm{u}+\bm{v}) - \bm{u} = \mathcal{P}_{\mathcal{K}-\bm{u}} (\bm{v}).
\end{equation}
\end{lemma}

\begin{lemma}
\label{lemma:cone_bound}
Let $\mathcal{D}$ and $\mathcal{C}$ be a nonempty and closed set and a closed cone, respectively, such that $\0 \in \mathcal{D}$ and $\mathcal{D} \subseteq \mathcal{C}$. Then for all $\bm{v} \in \mathbb{R}^n$
\begin{equation}
\label{Eq_cone_bound}
\| \mathcal{P}_{\mathcal{D}} (\bm{v}) \|_2 \leq \kappa \| \mathcal{P}_{\mathcal{C}} (\bm{v}) \|_2,
\end{equation}
where $\kappa=1$ if $\mathcal{D}$ is a convex set and $\kappa=2$ otherwise.
\end{lemma}

Let us now prove Theorem~\ref{theorem1}.
Since $s(\x_*)=R$, we have that 
${\x}_{t}-\x_*$ is inside the descent set $\mathcal{D}_{s}(\x_{*})$ for all $t$.
\tomt{Specifically, recall Definition~\ref{def:DnC} and note that $s(\x_*+({\x}_{t}-\x_*))=s({\x}_{t}) \leq R = s(\x_*)$, where the inequality uses the fact that ${\x}_{t} \in \mathcal{K} = \left \{ {\x} \in \mathbb{R}^n : s({\x}) \leq R \right \}$ by construction of the PGD.} 
For simplicity let us define $\mathcal{D} \triangleq \mathcal{D}_{s}(\x_{*})$ and $\mathcal{C} \triangleq \mathcal{C}_{s}(\x_{*})$.
We obtain \eqref{Eq_pgd_thm} by 
\begin{align}
\label{Eq_pgd_thm_proof}
\|{\x}_{t+1}-\x_*\|_2 &= \| \mathcal{P}_{\mathcal{K}} \left ( {\x}_{t} + \mu \W ( \y - \A {\x_t} ) \right ) - \x_*\|_2 \nonumber \\
& \stackrel{(a)}{=} \| \mathcal{P}_{\mathcal{K}-\x_*} \left ( {\x}_{t} + \mu \W ( \y - \A {\x_t} ) - \x_* \right ) \|_2 \nonumber \\
& \stackrel{(b)}{=} \| \mathcal{P}_{\mathcal{D}} \left ( {\x}_{t} + \mu \W ( \y - \A {\x_t} ) - \x_* \right ) \|_2 \nonumber \\
& \stackrel{(c)}{\leq} \kappa_s \| \mathcal{P}_{\mathcal{C}} \left ( {\x}_{t} + \mu \W ( \y - \A {\x_t} ) - \x_* \right ) \|_2 \nonumber \\
&= \kappa_s \| \mathcal{P}_{\mathcal{C}} \left ( (\I_n - \mu \W\A) ({\x}_{t} - \x_*) + \mu \W ( \y - \A \x_* ) \right ) \|_2 \nonumber \\
& \stackrel{(d)}{=} \kappa_s \supr{ \bm{v} \in \mathcal{C} \cap \mathbb{B}^n } \bm{v}^T  [ (\I_n - \mu \W\A) ({\x}_{t} - \x_*)  
+ \mu \W ( \y - \A \x_* )  ] \nonumber \\
& \leq \tomt{ \kappa_s \supr{ \bm{v} \in \mathcal{C} \cap \mathbb{B}^n } \bm{v}^T (\I_n - \mu \W\A) ({\x}_{t} - \x_*) 
+ \kappa_s \mu \supr{ \bm{v} \in \mathcal{C} \cap \mathbb{B}^n } \bm{v}^T \W ( \y - \A \x_* ) } \nonumber \\
& = \tomt{  \kappa_s \supr{ \bm{v} \in \mathcal{C} \cap \mathbb{B}^n } \bm{v}^T (\I_n - \mu \W\A) \frac{{\x}_{t} - \x_*}{\|{\x}_{t} - \x_*\|_2} \| {\x}_{t} - \x_* \|_2 + \kappa_s \mu \xi(\mathcal{C}) } \nonumber \\
& \stackrel{(e)}{\leq} \tomt{ \kappa_s \supr{ \bm{v},\bm{u} \in \mathcal{C} \cap \mathbb{B}^n } \bm{v}^T (\I_n - \mu \W\A) \bm{u}  \|{\x}_{t} - \x_*\|_2 + \kappa_s \mu \xi(\mathcal{C}) } \nonumber \\
& = \kappa_s \rho(\mathcal{C}) \|{\x}_{t} - \x_*\|_2 + \kappa_s \mu \xi(\mathcal{C}),
\end{align}
where $(a)$ follows from Lemma \ref{lemma:proj_shift} \tomt{(with $\bm{u}=\x_*$ and $\bm{v}={\x}_{t} + \mu \W ( \y - \A {\x_t} ) - \x_*$)}; $(b)$ follows from plugging $R=s(\x_*)$ in the definition of $\mathcal{K}$ (given in \eqref{Eq_set_Kr}); $(c)$ follows from Lemma \ref{lemma:cone_bound}; $(d)$ follows from Lemma \ref{lemma:cone_norm}; and $(e)$ 
\tomt{follows from $\frac{{\x}_{t} - \x_*}{\|{\x}_{t} - \x_*\|_2} \in \mathcal{D} \cap \mathbb{B}^n  \subseteq \mathcal{C} \cap \mathbb{B}^n$.}

\subsection{Proof of Proposition~\ref{prop_ls}}
\label{app:prop2}

For the LS objective, we have $\W=\A^T$ and $\mu_{LS} =  \| \nabla^2 \ell_{LS} \|^{-1} = \| \A^T\A \|^{-1}$. Therefore, $\I_n - \mu_{LS} \A^T \A$ is positive semi-definite, and using the generalized Cauchy-Schwarz inequality we get
\begin{align}
\label{Eq_pgd_rate_ls3}
\rho(\mathcal{C}_{s}(\x_{*})) &= \supr{ \bm{u},\bm{v} \in \mathcal{C}_{s}(\x_{*}) \cap \mathbb{S}^{n-1} } \bm{u}^T (\I_n - \mu_{LS} \A^T \A) \bm{v}  \nonumber \\
&\leq \supr{ \bm{u},\bm{v} \in \mathcal{C}_{s}(\x_{*}) \cap \mathbb{S}^{n-1} } \sqrt{ \bm{u}^T (\I_n - \mu_{LS} \A^T \A) \bm{u} } 
 \sqrt{ \bm{v}^T (\I_n - \mu_{LS} \A^T \A) \bm{v} } \nonumber \\
&= \supr{ \bm{u} \in \mathcal{C}_{s}(\x_{*}) \cap \mathbb{S}^{n-1} } \bm{u}^T (\I_n - \mu_{LS} \A^T \A) \bm{u}  \nonumber \\
&= 1 - \mu_{LS} \infim{ \bm{u} \in \mathcal{C}_{s}(\x_{*}) \cap \mathbb{S}^{n-1} } \| \A \bm{u} \|_2^2  \,\,\, = P_{LS}(\mathcal{C}_{s}(\x_{*})).
\end{align}

\subsection{Proof of Proposition~\ref{prop_bp}}
\label{app:prop3}

For the BP objective, we have $\W=\A^\dagger=\A^T(\A\A^T)^{-1}$ and $\mu_{BP} =  \| \nabla^2 \ell_{BP} \|^{-1} = \| \A^\dagger\A \|^{-1}=1$. 
Since $\I_n - \mu_{BP} \A^\dagger \A$ is positive semi-definite, using similar steps as those in \eqref{Eq_pgd_rate_ls3} we get
\begin{align}
\label{Eq_pgd_rate_bp3}
\rho(\mathcal{C}_{s}(\x_{*})) &= \supr{ \bm{u},\bm{v} \in \mathcal{C}_{s}(\x_{*}) \cap \mathbb{S}^{n-1} } \bm{u}^T (\I_n - \mu_{BP} \A^\dagger \A) \bm{v}  \nonumber \\
&\leq 1 - \mu_{BP} \infim{ \bm{u} \in \mathcal{C}_{s}(\x_{*}) \cap \mathbb{S}^{n-1} } \| (\A\A^T)^{-\frac{1}{2}}\A \bm{u} \|_2^2  \nonumber \\ 
& = P_{BP}(\mathcal{C}_{s}(\x_{*})).
\end{align}

\section{Numerical Experiments Demonstrating $P_{BP} < P_{LS}$ (Strict Inequality)}
\label{app:conjecture}

\begin{figure}[t]
  \centering
  \begin{subfigure}[b]{0.5\linewidth}
    \centering\includegraphics[width=150pt]{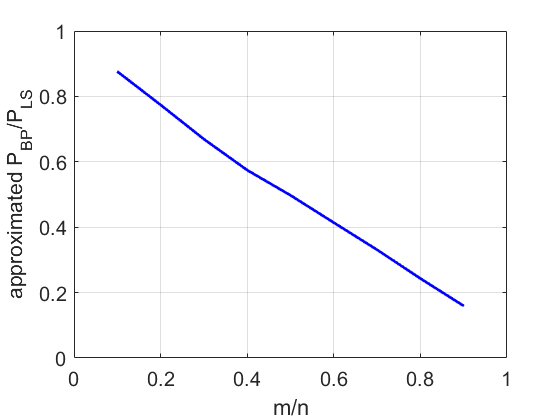}
    \caption{\label{fig:approx_ratio_vs_m}}
  \end{subfigure}%
  \begin{subfigure}[b]{0.5\linewidth}
    \centering\includegraphics[width=150pt]{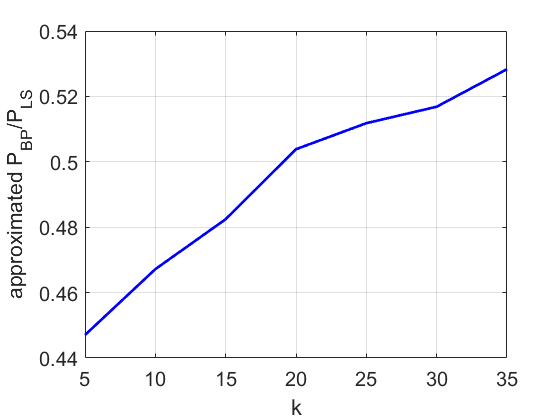}
    \caption{\label{fig:approx_ratio_vs_k}}
  \end{subfigure}
  \caption{Approximate ratio ${P}_{BP}/{P}_{LS}$ for $k$-sparse $\u \in \mathbb{R}^{1024}$ and Gaussian $\A$ (see the text for details) for: (\subref{fig:approx_ratio_vs_m}) $k=20$ and $m$ is varied; (\subref{fig:approx_ratio_vs_k}) $m=512$ and $k$ is varied.}
\label{fig:approx_ratio}
\end{figure}

In this section, we present 
experiments that support 
our conjecture from Section~\ref{sec:convergence_intrin}:
The inequality in Proposition \ref{prop1} is strict, i.e., that $P_{BP}(\mathcal{C}_{s}(\x_{*})) < P_{LS}(\mathcal{C}_{s}(\x_{*}))$, in generic cases when the entries of  
$\A \in \mathbb{R}^{m \times n}$ are i.i.d Gaussians $\mathcal{N}(0,\frac{1}{m})$, the recovered signals belong to parsimonious models and feasible sets are appropriately chosen.

We consider a Gaussian $\A$, as mentioned in Section~\ref{sec:convergence_intrin}, and $\mathcal{C}$ which is the set of $k$-sparse signals, i.e., the number of non-zero elements in any $\u \in \mathcal{C}$ is at most $k$.
In this case, $\infim{ \bm{u} \in \mathcal{C} \cap \mathbb{S}^{n-1} } \|\A \bm{u} \|_2^2$ can be approximated by: 1) drawing many supports, i.e., choices of $k$ out of the $n$ columns of $\A$; 2) for each support creating an $m \times k$ matrix $\tilde{\A}$ and computing $\sigma_{min}(\tilde{\A}^T\tilde{\A})$; and 3) keeping the minimal value.
Plugging the approximation of $\infim{ \bm{u} \in \mathcal{C} \cap \mathbb{S}^{n-1} } \|\A \bm{u} \|_2^2$ in \eqref{Eq_pgd_rate_ls}, we obtain an approximation of $P_{LS}(\mathcal{C})$.

Similarly, to approximate $\infim{ \bm{u} \in \mathcal{C} \cap \mathbb{S}^{n-1} } \|(\A\A^T)^{-\frac{1}{2}}\A \bm{u} \|_2^2$ the same procedure can be done with $\sigma_{min}(\tilde{\A}^T(\A\A^T)^{-1}\tilde{\A})$. 
Plugging the approximation of $\infim{ \bm{u} \in \mathcal{C} \cap \mathbb{S}^{n-1} } \|(\A\A^T)^{-\frac{1}{2}}\A \bm{u} \|_2^2$ in \eqref{Eq_pgd_rate_bp}, we obtain an approximation of $P_{BP}(\mathcal{C})$.

Fig.~\ref{fig:approx_ratio_vs_m} shows the approximate ratio $\hat{P}_{BP}/\hat{P}_{LS}$ for $n=1024, k=20$ and different values of $m$. Fig.~\ref{fig:approx_ratio_vs_k} shows this ratio for $n=1024, m=512$ and different values of $k$. In both figures $\hat{P}_{BP}$ is {\em strictly} smaller than $\hat{P}_{LS}$, which agrees with our conjecture.

\section{More Details on Condition~\ref{cond3}}
\label{app:Contraction}

As explained in Section~\ref{sec:convergence_beyond}, the non-expansive property that is stated in \eqref{Eq_non_expansive} is satisfied by the proximal mapping of any convex function \cite{beck2017first}.
However, this property is not enough for distinguishing between the convergence rates of the proximal gradient method \eqref{Eq_ista} for the LS and BP terms. 
Therefore, stronger conditions on $\beta s(\cdot)$ are required. 
One such condition is that the proximal mapping of $\beta s(\cdot)$ is a contraction, i.e., there exists $0 \leq  k_{\beta s(\cdot)} < 1$ such that for all ${\z}_1, {\z}_2$
\begin{align}
\label{Eq_contraction_full}
\| \mathrm{prox}_{\beta s(\cdot)}({\z}_1) - \mathrm{prox}_{\beta s(\cdot)}({\z}_2) \|_2 \leq k_{\beta s(\cdot)} \| {\z}_1 - {\z}_2 \|_2. 
\end{align}
Note that even though the above condition is rather strict, it is satisfied by some prior functions such as Tikhonov regularization \cite{tikhonov1963solution} (where $s({\x})=\frac{1}{2}\|\D{\x}\|_2^2$ and $\D^T\D$ is positive definite) or even a recent GMM-based prior \cite{teodoro2018convergent} (see Lemma 2 there).

Condition~\ref{cond3}, which is required in Theorem~\ref{theorem_ista}, is less demanding than \eqref{Eq_contraction_full}. 
Specifically, 
satisfying \eqref{Eq_contraction_full} with $k_{\beta s(\cdot)}$ implies satisfying \eqref{Eq_contraction} with $\delta_{\A,\beta s(\cdot)}=1-k_{\beta s(\cdot)}$. This is a simple consequence of the Pythagorean theorem and the fact that $0 \leq k_{\beta s(\cdot)}<1$
\begin{align}
\label{Eq_contraction_full2}
\| \mathrm{prox}_{\beta s(\cdot)}({\z}_1) - \mathrm{prox}_{\beta s(\cdot)}({\z}_2) \|_2^2 
& \leq k_{\beta s(\cdot)}^2 \| {\z}_1 - {\z}_2 \|_2^2  \nonumber \\
& = k_{\beta s(\cdot)}^2 \left ( \| \P_A({\z}_1 - {\z}_2) \|_2^2 + \| \Q_A({\z}_1 - {\z}_2) \|_2^2 \right ) \nonumber \\
& \leq  \| \P_A({\z}_1 - {\z}_2) \|_2^2 + \| k_{\beta s(\cdot)} \Q_A({\z}_1 - {\z}_2) \|_2^2  \nonumber \\
& =  \| (\P_A + k_{\beta s(\cdot)} \Q_A)({\z}_1 - {\z}_2) \|_2^2. 
\end{align}
Therefore, priors that satisfy \eqref{Eq_contraction_full} (e.g., \cite{tikhonov1963solution,teodoro2018convergent}) also satisfy Condition~\ref{cond3}.

Another property of Condition~\ref{cond3} relates to the effect of the regularization parameter $\beta$ on $\delta_{\A,\beta s(\cdot)}$. 
Note that for $\beta_1 \geq \beta_2$ we have that the weight of the prior $s(\cdot)$ in the proximal mapping \eqref{def_prox} is larger for $\mathrm{prox}_{\beta_1 s(\cdot)}({\z})$ than for $\mathrm{prox}_{\beta_2 s(\cdot)}({\z})$.
Therefore, it is expected to impose a stronger restriction on the null space of $\A$, or equivalently $\delta_{\A,\beta_1 s(\cdot)} \geq \delta_{\A,\beta_2 s(\cdot)}$.

\section{Proof of Theorem~\ref{theorem_ista}}
\label{app:proof_ista}

In this section we prove Theorem~\ref{theorem_ista}. 
The existence of the stationary point $\x_*=\mathrm{prox}_{\mu \beta s(\cdot)}(\x_* - \mu \nabla  \ell(\x_*))$ (that is a minimizer of \eqref{Eq_cost_func_general}) to which proximal gradient descent with step-size $\mu=1/\tilde{\sigma}_{max}$ converges follows from the  
convergence result 
in \cite{beck2009fast}. Yet, this result guarantees only sub-linear convergence. In the following we obtain the desired linear convergence result.
\begin{align}
\label{Eq_ista_thm_proof}
\|{\x}_{t+1}-\x_*\|_2 
&= \| \mathrm{prox}_{\mu \beta s(\cdot)}({\x}_{t} - \mu \nabla  \ell({\x}_{t})) -  \mathrm{prox}_{\mu \beta s(\cdot)}(\x_* - \mu \nabla  \ell(\x_*)) \|_2 \nonumber \\
& \stackrel{(a)}{\leq} \big \| \left (\P_A + (1-\delta_{\A,\mu\beta s(\cdot)})\Q_A \right ) 
( ({\x}_{t} - \mu \nabla  \ell({\x}_{t})) - (\x_* - \mu \nabla  \ell(\x_*)) ) \big \|_2 \nonumber \\
& \stackrel{(b)}{=} \big \| \big ( (\P_A + (1-\delta_{\A,\mu\beta s(\cdot)})\Q_A ) {\x}_{t} - \mu \nabla  \ell({\x}_{t})  \big ) \nonumber \\ 
& \hspace{10mm} 
- \big ( (\P_A + (1-\delta_{\A,\mu\beta s(\cdot)})\Q_A ) \x_* - \mu \nabla  \ell(\x_*)  \big ) \big \|_2  \nonumber \\
& \stackrel{(c)}{=} \left \| \g({\x}_t) - \g(\x_*) \right \|_2,
\end{align}
where $(a)$ follows from Condition~\ref{cond3}; $(b)$ follows from the assumption that $\nabla\ell(\cdot) \in \mathrm{range}(\A^T)$, which implies $\P_A \nabla\ell(\cdot) = \nabla\ell(\cdot)$ and $\Q_A \nabla\ell(\cdot) =0$; and $(c)$ uses the definition
$\g({\x}) \triangleq (\P_A + (1-\delta_{\A,\mu\beta s(\cdot)})\Q_A ) {\x} - \mu \nabla  \ell({\x})$. 

Using Taylor series expansion consideration, there exists a point $\bxi$ in the line between ${\x}_{t}$ and $\x_*$, such that $\| \g({\x}_t) - \g(\x_*) \|_2 = \| \nabla \g(\bxi) ( {\x}_t  - \x_* ) \|_2$ . Therefore,
\begin{align}
\label{Eq_ista_thm_proof_2}
\| \g({\x}_t) - \g(\x_*) \|_2 & \leq \| \nabla \g(\bxi) \| \| {\x}_t  - \x_* \|_2 \nonumber \\ 
&  \leq  \maxim{\tilde{\bxi}} \| \nabla \g(\tilde{\bxi}) \| \| {\x}_t  - \x_* \|_2,
\end{align}
which yields
\begin{align}
\label{Eq_ista_thm_proof_3}
\|{\x}_{t+1}-\x_*\|_2 \leq  \maxim{\tilde{\bxi}} \| \nabla \g(\tilde{\bxi}) \| \| {\x}_t  - \x_* \|_2.
\end{align}
Note that 
\begin{align}
\nabla \g({\x}) = \P_A - \mu \nabla^2  \ell({\x}) + (1-\delta_{\A,\mu\beta s(\cdot)})\Q_A.
\end{align}
Therefore, we have
\begin{align}
\maxim{\tilde{\bxi}} \| \nabla \g(\tilde{\bxi}) \| 
 = \mathrm{max} \left \{ |1-\mu \tilde{\sigma}_{max}|, |1-\mu \tilde{\sigma}_{min}| , 1 - \delta_{\A,\mu\beta s(\cdot)}  \right \}.
\end{align}
Recall that $\tilde{\sigma}_{max}$ is the largest eigenvalue of $\nabla^2 \ell(\cdot)$ and $\tilde{\sigma}_{max}$ is the smallest {\em non-zero} eigenvalue of $\nabla^2 \ell(\cdot)$.
For the widely used step-size $\mu=1/\tilde{\sigma}_{max}$ considered in the theorem we get 
$\maxim{\tilde{\bxi}} \| \nabla \g(\tilde{\bxi}) \| = \mathrm{max} \left \{ 0, 1-\frac{\tilde{\sigma}_{min}}{\tilde{\sigma}_{max}} , 1 - \delta_{\A,\frac{\beta}{\tilde{\sigma}_{max}}s(\cdot)}  \right \}$. 
Finally, plugging this in \eqref{Eq_ista_thm_proof_3} yields \eqref{Eq_ista_thm}.

\section{The Connection Between ALISTA and IDBP with $\ell_1$-Norm Prior} 
\label{app:alista}

As pointed out in \cite{tirer2019back}, the IDBP algorithm \cite{tirer2019image} is essentially the proximal gradient method \eqref{Eq_ista}, applied on $\ell_{BP}({\x})+\beta s({\x})$. 
For the special case where $s({\x})=\|{\x}\|_1$, we have that the proximal mapping 
$\mathrm{prox}_{\mu \beta s(\cdot)}({\z}) = \argmin{{\x}} \, \frac{1}{2} \| {\z} - {\x} \|_2^2 + \mu \beta s({\x})$
is the soft-thresholding operator
\begin{align}
\label{def_prox_ell1}
\mathrm{prox}_{\mu \beta s(\cdot)}({\z}) = \mathcal{T}_{\mu \beta}({\z}),
\end{align}
where $[\mathcal{T}_{\theta}({\z})]_i=\mathrm{sign}({z}_i)\mathrm{max}(|{z}_i|-\theta,0)$.
For the BP term, one can set the step-size $\mu_{BP}=1$. 
Recalling 
$\nabla \ell_{BP}({\x})$ given in \eqref{Eq_fidelity_grads}, we get the $\ell_1$-IDBP
\begin{align}
\label{Eq_ista_ell1_bp}
{\x}_{t+1} =  \mathcal{T}_{\beta}({\x}_{t} - \A^\dagger ( \A {\x}_{t} - \y )).
\end{align}

When using the traditional LS rather than the BP term, instead of \eqref{Eq_ista_ell1_bp}, one gets the popular ISTA algorithm \cite{daubechies2004iterative} (with step-size $\mu_{LS}$)
\begin{align}
\label{Eq_ista_ell1_ls}
{\x}_{t+1} =  \mathcal{T}_{\mu_{LS}\beta}({\x}_{t} - \mu_{LS} \A^T ( \A {\x}_{t} - \y )).
\end{align}
Since ISTA requires a large number of iterations, the seminal LISTA paper \cite{gregor2010learning}
suggested to reduce the computational complexity by unrolling a few ISTA iterations and learning (offline) the linear operators and the soft-thresholds for each iteration that will give the desired result.

The ALISTA paper \cite{liu2019alista} demonstrated that, for very sparse signals, similar convergence rate as LISTA can be obtained by the following scheme (equation (15) in \cite{liu2019alista})
\begin{align}
\label{Eq_alista}
{\x}_{t+1} =  \mathcal{T}_{\theta_t}({\x}_{t} - \mu_{t} \tilde{\W}^T ( \A {\x}_{t} - \y )),
\end{align}
where only the soft-threshold $\theta_t$ and the step-size $\mu_{t}$ are learned for each iteration, while the matrix $\tilde{\W}$ is analytically obtained.
In the implementation of ALISTA (and its follow-up papers \cite{wu2019sparse,behrens2020neurally}), they obtain $\tilde{\W}$ by numerical minimization of the following problem (see equation (16) in \cite{liu2019alista} and Appendix~E.1 there):
\begin{align}
\label{Eq_alista_W}
    \tilde{\W} = &\argmin{\overline{\W}} \, \| \overline{\W}^T \A \|_F^2 
    \,\,\,\,\,\, 
    \mathrm{s.t.} \,\,\, \overline{\W}[:,i]^T\A[:,i]=1, \,\,\,\, 1\leq i \leq n.
\end{align}
However, they 
have not used 
the fact that \eqref{Eq_alista_W} has a closed-form solution\footnote{It is obtained by observing that each column of $\overline{\W}$ in \eqref{Eq_alista_W} can be optimized independently.} given by
\begin{align}
\label{Eq_alista_W_sol}
    \tilde{\W} = (\A\A^T)^{-1}\A\bLambda = (\A^\dagger)^T \bLambda,
\end{align}
where $\bLambda$ is an $n \times n$ diagonal matrix of values given by $\{ \lambda_i = \left ( \A[:,i]^T(\A\A^T)^{-1}\A[:,i] \right )^{-1} \}_{i=1}^n$. 
Therefore, 
the ALISTA iteration may be read as
\begin{align}
\label{Eq_alista2}
{\x}_{t+1} =  \mathcal{T}_{\theta_t}({\x}_{t} - \mu_{t} \bLambda \A^\dagger ( \A {\x}_{t} - \y )).
\end{align}
Moreover, observing the learned values of $\{ \mu_{t} \}$ and $\{ \theta_t \}$ for Gaussian compressed sensing (CS) with $m/n=0.5$ (Figure 2 in \cite{liu2019alista}), we see that $\mu_{t}$ fluctuates around 1, while $\theta_t$ monotonically decreases. Note that for $\mu_{t}=1$ we have that decreasing $\theta_t$ is equivalent to decreasing the regularization parameter $\beta$ (starting from a large value for $\beta$ and monotonically decreasing it, is a known way to accelerate 
sparse coding algorithms, e.g., see \cite{jiao2017iterative,zarka2019deep}).
Therefore,  
in essence, the only key difference between ALISTA \eqref{Eq_alista2} and the $\ell_1$-IDBP \eqref{Eq_ista_ell1_bp} is the matrix $\bLambda$ that only normalizes the rows of $\A^\dagger$.
Without this matrix, \eqref{Eq_alista2} is simply $\ell_1$-IDBP with per-iteration tuning of $\beta$.

\begin{figure}[t]
  \centering
    \centering\includegraphics[width=200pt]{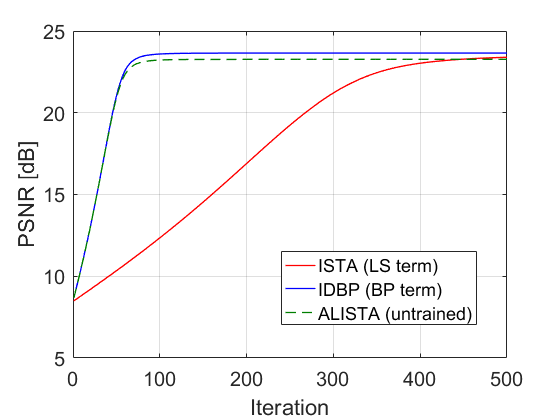}
  \caption{Compressed sensing with $m/n=0.5$ Gaussian measurements and SNR of 20dB.
   PSNR (averaged over 4 test images) versus iteration number, for ISTA (based on LS), untrained ALISTA, and $\ell_1$-IDBP (based on BP), with $\beta=4.64$.}
\label{fig:CS_results_with_alista}
\end{figure}

Fig.~\ref{fig:CS_results_with_alista} shows 
the results of the three algorithms: ISTA (with the typical  step-size $\mu_{LS} = 1/\sigma_{max}(\A\A^T)$), untrained ALISTA (with $\mu_{t}=1$ and $\theta_t=\beta$), and $\ell_1$-IDBP (with $\mu_{BP}=1$). We consider the same Gaussian CS scenario that has led to Fig.~\ref{fig:CS_results_fista} (for FISTA) in Section~\ref{sec:exp_l1}. 
The columns of $\A$ are normalized to have the unit $\ell_2$ norm, as assumed in \cite{liu2019alista}. 
All the algorithms are initialized with ${\x}_0=\A^\dagger\y$ and use $\beta=4.64$ that was shown to be good for all the algorithms in Fig.~\ref{fig:CS_results_fista}.
Note that the curves of the untrained ALISTA and $\ell_1$-IDBP 
in Fig.~\ref{fig:CS_results_with_alista} 
have almost identical convergence rate,  
which means that $\bLambda$ in ALISTA is not significant in this case. 
Both of them 
are much faster than ISTA (which is based on the LS term).
Thus, similar to IDBP, ALISTA enjoys the benefits of the BP term, 
which provides it a good starting point and allows it to reach very fast convergence in \cite{liu2019alista}
with only a very small number of parameters that are learned.

As for the existing analysis of ALISTA, note that it is based on the notion of mutual coherence that is restricted to sparse signals and yields over-pessimistic guarantees (e.g., for a Gaussian matrix $\A$ of size $250 \times 500$, the generalized mutual coherence of $\tilde{\W}$, denoted by $\tilde{\mu}$,  
is typically larger than 0.2, which implies that the convergence theorem in \cite{liu2019alista} holds only for signals in $\mathbb{R}^{500}$ with less than $(1+1/\tilde{\mu})/2 < 3$ non-zero elements). On the other hand, the ``restricted smallest eigenvalue" (restricted to the set $\mathcal{C}_{s}(\x_{*})$) analysis used in our paper 
allows less demanding sparsity levels and 
covers more general low-dimensional signal models. 
Prior works also showed different cases where ``restricted smallest eigenvalue" based analysis provided tight phase transitions and bounds for the LS objective (e.g., see the discussions in Section~\ref{sec:convergence}).

\section{Fundamental Differences Between Preconditioning and the BP Approach}
\label{app:precond}

It is important to 
note that using the BP term for ill-posed problem (i.e., minimizing $f_{BP}({\x})=\ell_{BP}({\x})+\beta s({\x})$ rather than $f_{LS}({\x})=\ell_{LS}({\x})+\beta s({\x})$) is fundamentally different than
the technique of preconditioning 
(discussed in unconstrained optimization literature, e.g., see \cite{saad2003iterative}) 
that 
is used to accelerate gradient-based methods {\em without changing} the considered objective function (and its minimizer).
In more detail, 
in order to perform preconditioning when minimizing an objective $f({\x})$, one multiplies its gradient $\nabla f({\x})$ by an invertible matrix $\P^{-1}$ with the goal of facilitating the optimization space curvature (taking this to the extreme, we have the Newton method where $\P$ is actually adaptive and equals the Hessian $\nabla^2 f({\x})$). Note that $\P^{-1}\nabla f({\x})=0 \Longleftrightarrow \nabla f({\x})=0$, and so, preconditioning does not change the minimizers.
On the other hand, the minimizers of $f_{BP}({\x})$ and $f_{LS}({\x})$ are different, e.g., see in \cite{tirer2019back} their different closed-form solutions when $s({\x})$ is Tikhonov regularization.

Another key point is that existing acceleration results for preconditioning require strong convexity of the objective. However, since we consider ill-posed problems, $\ell_{LS}({\x})$ and $\ell_{BP}({\x})$ are not strongly convex and so are the most widely used priors $s({\x})$ (such as those considered in this paper).
Therefore, previous preconditioning results do not explain the faster convergence observed when using the BP term 
instead of 
the LS term. In fact, our study reveals that 
this advantage 
depends on the amount of restrictions that the prior $s({\x})$ (implicitly) imposes on ${\x}$ in the null space of $\A$. 
For example, as preconditioning theory considers well-posed problems, it states that the more bad-conditioned a linear system $\A$ is, the higher the acceleration that can be obtained by the preconditioning. 
However, our experiments (e.g., see Figs.~\ref{fig:CS_results_dcgan} and \ref{fig:SR_results_dcgan} in Section~\ref{sec:exp_dcgan}) 
show that the convergence advantage of BP over LS for super-resolution (where $\A\A^T$ is badly conditioned) can be smaller than 
for compressed sensing (where $\A\A^T$ is much better conditioned). With our analysis, this can be explained by a weaker restriction that the prior imposes on the null space in the super-resolution case.

\section{More Details on the Considered Step-Size}
\label{app:StepSize}

As discussed in Section~\ref{sec:convergence}, in this paper we examine optimization schemes with step-size of 1 over the Lipschitz constant $\nabla \ell(\cdot)$. 
This step-size is the most common choice of practitioners.

To explain the popularity of this step-size, let us consider the minimization of a general convex function $\ell(\cdot): \mathbb{R}^n \rightarrow \mathbb{R}$ (without any additional prior term). Let $L$ be the Lipschitz constant of $\nabla \ell(\cdot)$, i.e., $\| \nabla \ell(\x_2) - \nabla \ell(\x_1) \|_2 \leq L \| \x_2 - \x_1 \|_2$ for all $\x_2, \x_1$. Equivalently  \cite{beck2017first}, this implies that for all $\x_2, \x_1$ we have 
\begin{align}
\label{Eq_smoothness_cond}
\ell(\x_2) - \ell(\x_1) \leq \nabla \ell(\x_1)^T (\x_2 - \x_1) + \frac{L}{2} \| \x_2 - \x_1 \|_2^2.
\end{align}
Now, recall that a gradient descent iteration for minimizing $\ell(\cdot)$ with step-size $\mu$ is given by 
\begin{align}
\label{Eq_gd_iter}
{\x}_{t+1} = {\x}_{t} - \mu \nabla \ell({\x}_{t}).
\end{align}
Using $\x_1={\x}_{t}$ and $\x_2={\x}_{t+1}$ in \eqref{Eq_smoothness_cond}, we get
\begin{align}
\label{Eq_gd_iter2}
\ell({\x}_{t+1}) - \ell({\x}_{t}) \leq -\mu \| \nabla \ell({\x}_{t}) \|_2^2 + \mu^2 \frac{L}{2} \| \nabla \ell({\x}_{t}) \|_2^2.
\end{align}
Note that the right-hand side (RHS) of \eqref{Eq_gd_iter2} is a simple parabola in $\mu$. 
Therefore, while any constant step-size $\mu \in (0,\frac{2}{L})$ ensures the convergence of gradient descent, the choice $\mu=\frac{1}{L}$ is optimal --- it minimizes the RHS of \eqref{Eq_gd_iter2}. 
We also note that, as far as we know, the guarantees for convergence rates of $\mathcal{O}(\frac{1}{t})$ and $\mathcal{O}(\frac{1}{t^2})$ for plain and accelerated proximal gradient algorithms, respectively, applied on (non-strongly) convex functions require that a {\em constant} step-size obeys $\mu \leq \frac{1}{L}$ \cite{beck2009fast,beck2017first}.

\begin{table}[h]
\small
\renewcommand{\arraystretch}{1.3}
\caption{PSNR [dB] 
(averaged over 10 test images) of PGD with 50 iterations and 
DCGAN prior for compressed sensing with $m/n=0.5$ Gaussian measurements and SNR of 20dB.} \label{table:dcgan2}
\centering
    \begin{tabular}{ | l | l | l |}
    \hline
            &  LS objective & BP objective \\ \hline
  CS $m/n=0.5$  & 23.14 & 23.57 \\ \hline
    \end{tabular}
\end{table}

\begin{table}[h]
\small
\renewcommand{\arraystretch}{1.3}
\caption{PSNR [dB] 
(averaged over 10 test images) of PGD with 50 iterations and 
DCGAN prior for super-resolution with Gaussian filter and scale factor of 3.} \label{table:dcgan2b}
\centering
    \begin{tabular}{ | l | l | l |}
    \hline
            &  LS objective & BP objective \\ \hline
  SR x3  & 23.29 & 23.90 \\ \hline
    \end{tabular}
\end{table}

\begin{figure}[h]
 \centering
  \subcaptionbox*{{ img. 202592}  }{%
  \includegraphics[width=0.15\columnwidth]{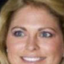}}
  \subcaptionbox*{{ $\A^\dagger\y$} }{%
  \includegraphics[width=0.15\columnwidth]{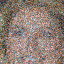}}
  \subcaptionbox*{{ LS (24.75 dB)} }{%
  \includegraphics[width=0.15\columnwidth]{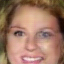}}
  \subcaptionbox*{{ BP (24.76 dB)} }{%
  \includegraphics[width=0.15\columnwidth]{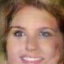}}\\
  \vspace{2mm}
  \subcaptionbox*{{ img. 202597 } }{%
  \includegraphics[width=0.15\columnwidth]{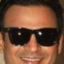}}
  \subcaptionbox*{{ $\A^\dagger\y$} }{%
  \includegraphics[width=0.15\columnwidth]{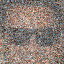}}
  \subcaptionbox*{{ LS (23.56 dB)} }{%
  \includegraphics[width=0.15\columnwidth]{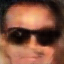}}
  \subcaptionbox*{{ BP (24.00 dB)} }{%
  \includegraphics[width=0.15\columnwidth]{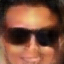}}\\
\caption{Compressed sensing with $m/n=0.5$ Gaussian measurements and SNR of 20dB, using PGD with 50 iterations and DCGAN prior. 
}\label{fig:CS_dcgan2}

\vspace{10mm}

 \centering
  \subcaptionbox*{{ img. 202596}  }{%
  \includegraphics[width=0.15\columnwidth]{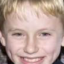}}
  \subcaptionbox*{{ Bicubic} }{%
  \includegraphics[width=0.15\columnwidth]{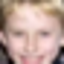}}
  \subcaptionbox*{{ LS (23.24 dB)} }{%
  \includegraphics[width=0.15\columnwidth]{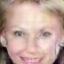}}
  \subcaptionbox*{{ BP (24.03 dB)} }{%
  \includegraphics[width=0.15\columnwidth]{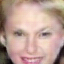}}\\
  \vspace{2mm}  
  \subcaptionbox*{{ img. 202598}  }{%
  \includegraphics[width=0.15\columnwidth]{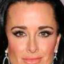}}
  \subcaptionbox*{{ Bicubic} }{%
  \includegraphics[width=0.15\columnwidth]{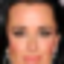}}
  \subcaptionbox*{{ LS (23.48 dB)} }{%
  \includegraphics[width=0.15\columnwidth]{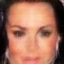}}
  \subcaptionbox*{{ BP (24.62 dB)} }{%
  \includegraphics[width=0.15\columnwidth]{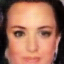}}\\
\caption{Super-resolution with Gaussian filter and scale factor of 3, using PGD with 50 iterations and DCGAN prior. 
}\label{fig:SR_dcgan}
\end{figure}

\section{Quantitative and Visual  
Results for PGD with DCGAN Prior}
\label{app:visual}

In this section we present quantitative results (average PSNR), as well as several visual results, which are obtained for the experiments in Section~\ref{sec:exp_dcgan}.

For the compressed sensing experiments (described in the main body of the paper), Table~\ref{table:dcgan2} shows the PSNR of the reconstructions, averaged over the test images. Several visual results are shown in Fig.~\ref{fig:CS_dcgan2}.

For the super-resolution experiments (described in the main body of the paper), Table~\ref{table:dcgan2b} shows the PSNR of the reconstructions, averaged over the test images. Several visual results are shown in Fig.~\ref{fig:SR_dcgan}.

\section*{Acknowledgements}
The authors would like to thank Amir Beck for fruitful discussions.

\bibliographystyle{siamplain}
\bibliography{final_article}

\end{document}

%% file: ex_shared.tex
\usepackage{lipsum}
\usepackage{amsfonts}
\usepackage{graphicx}
\usepackage{epstopdf}
\usepackage{algorithmic}
\ifpdf
  \DeclareGraphicsExtensions{.eps,.pdf,.png,.jpg}
\else
  \DeclareGraphicsExtensions{.eps}
\fi

\usepackage{enumitem}
\setlist[enumerate]{leftmargin=.5in}
\setlist[itemize]{leftmargin=.5in}

\newsiamremark{remark}{Remark}
\newsiamremark{hypothesis}{Hypothesis}
\crefname{hypothesis}{Hypothesis}{Hypotheses}
\newsiamthm{claim}{Claim}

\headers{On the Convergence Rate of PGD for a Back-Projection based Objective}{T. Tirer and R. Giryes}

\title{On the Convergence Rate of Projected Gradient Descent \\ for a Back-Projection based Objective\thanks{
Accepted to SIAM Journal on Imaging Sciences (SIIMS).
\funding{This research is supported by ERC-StG grant no. 757497 (SPADE) and gifts from NVIDIA, Amazon, and Google.}}}

\author{Tom Tirer\thanks{School of Electrical Engineering, Tel Aviv University, Tel Aviv, Israel 
  (\email{tirer.tom@gmail.com}).}
\and Raja Giryes\thanks{School of Electrical Engineering, Tel Aviv University, Tel Aviv, Israel 
  (\email{raja@tauex.tau.ac.il}).}}

\usepackage{amsopn}


%% file: final_article.bbl
\begin{thebibliography}{10}

\bibitem{abu2021image}
{\sc S.~Abu-Hussein, T.~Tirer, S.~Y. Chun, Y.~C. Eldar, and R.~Giryes}, {\em
  Image restoration by deep projected {GSURE}}, arXiv preprint
  arXiv:2102.02485,  (2021).

\bibitem{shady2019image}
{\sc S.~Abu~Hussein, T.~Tirer, and R.~Giryes}, {\em Image-adaptive {GAN} based
  reconstruction}, in Proceedings of the AAAI Conference on Artificial
  Intelligence, vol.~34, 2020, pp.~3121--3129.

\bibitem{afonso2010augmented}
{\sc M.~V. Afonso, J.~M. Bioucas-Dias, and M.~A. Figueiredo}, {\em An augmented
  lagrangian approach to the constrained optimization formulation of imaging
  inverse problems}, IEEE Transactions on Image Processing, 20 (2010),
  pp.~681--695.

\bibitem{amelunxen2014living}
{\sc D.~Amelunxen, M.~Lotz, M.~B. McCoy, and J.~A. Tropp}, {\em Living on the
  edge: Phase transitions in convex programs with random data}, Information and
  Inference: A Journal of the IMA, 3 (2014), pp.~224--294.

\bibitem{beck2017first}
{\sc A.~Beck}, {\em First-order methods in optimization}, vol.~25, SIAM, 2017.

\bibitem{beck2009fast}
{\sc A.~Beck and M.~Teboulle}, {\em A fast iterative shrinkage-thresholding
  algorithm for linear inverse problems}, SIAM journal on imaging sciences, 2
  (2009), pp.~183--202.

\bibitem{behrens2020neurally}
{\sc F.~Behrens, J.~Sauder, and P.~Jung}, {\em Neurally augmented {ALISTA}},
  arXiv preprint arXiv:2010.01930,  (2020).

\bibitem{biemond1990iterative}
{\sc J.~Biemond, R.~L. Lagendijk, and R.~M. Mersereau}, {\em Iterative methods
  for image deblurring}, Proceedings of the IEEE, 78 (1990), pp.~856--883.

\bibitem{bora2017compressed}
{\sc A.~Bora, A.~Jalal, E.~Price, and A.~G. Dimakis}, {\em Compressed sensing
  using generative models}, in Proceedings of the 34th International Conference
  on Machine Learning-Volume 70, JMLR. org, 2017, pp.~537--546.

\bibitem{candes2008introduction}
{\sc E.~J. Cand{\`e}s and M.~B. Wakin}, {\em An introduction to compressive
  sampling}, IEEE signal processing magazine, 25 (2008), pp.~21--30.

\bibitem{chandrasekaran2012convex}
{\sc V.~Chandrasekaran, B.~Recht, P.~A. Parrilo, and A.~S. Willsky}, {\em The
  convex geometry of linear inverse problems}, Foundations of Computational
  mathematics, 12 (2012), pp.~805--849.

\bibitem{dabov2007image}
{\sc K.~Dabov, A.~Foi, V.~Katkovnik, and K.~Egiazarian}, {\em Image denoising
  by sparse {3-D} transform-domain collaborative filtering}, IEEE Transactions
  on image processing, 16 (2007), pp.~2080--2095.

\bibitem{danielyan2012bm3d}
{\sc A.~Danielyan, V.~Katkovnik, and K.~Egiazarian}, {\em B{M3D} frames and
  variational image deblurring}, IEEE Transactions on Image Processing, 21
  (2012), pp.~1715--1728.

\bibitem{daubechies2004iterative}
{\sc I.~Daubechies, M.~Defrise, and C.~De~Mol}, {\em An iterative thresholding
  algorithm for linear inverse problems with a sparsity constraint},
  Communications on Pure and Applied Mathematics: A Journal Issued by the
  Courant Institute of Mathematical Sciences, 57 (2004), pp.~1413--1457.

\bibitem{donoho1995noising}
{\sc D.~L. Donoho}, {\em De-noising by soft-thresholding}, IEEE transactions on
  information theory, 41 (1995), pp.~613--627.

\bibitem{donoho2006compressed}
{\sc D.~L. Donoho}, {\em Compressed sensing}, IEEE Transactions on information
  theory, 52 (2006), pp.~1289--1306.

\bibitem{duchi2008efficient}
{\sc J.~Duchi, S.~Shalev-Shwartz, Y.~Singer, and T.~Chandra}, {\em Efficient
  projections onto the $\ell_1$-ball for learning in high dimensions}, in
  Proceedings of the 25th international conference on Machine learning, ACM,
  2008, pp.~272--279.

\bibitem{eldar2008generalized}
{\sc Y.~C. Eldar}, {\em Generalized {SURE} for exponential families:
  Applications to regularization}, IEEE Transactions on Signal Processing, 57
  (2008), pp.~471--481.

\bibitem{genzel2017ell}
{\sc M.~Genzel, G.~Kutyniok, and M.~M{\"a}rz}, {\em $\ell_1$-analysis
  minimization and generalized (co-) sparsity: When does recovery succeed?},
  arXiv preprint arXiv:1710.04952,  (2017).

\bibitem{goodfellow2014generative}
{\sc I.~Goodfellow, J.~Pouget-Abadie, M.~Mirza, B.~Xu, D.~Warde-Farley,
  S.~Ozair, A.~Courville, and Y.~Bengio}, {\em Generative adversarial nets}, in
  Advances in neural information processing systems, 2014, pp.~2672--2680.

\bibitem{gordon1988milman}
{\sc Y.~Gordon}, {\em On {M}ilman's inequality and random subspaces which
  escape through a mesh in $\mathbb{R}^n$}, in Geometric Aspects of Functional
  Analysis, Springer, 1988, pp.~84--106.

\bibitem{gregor2010learning}
{\sc K.~Gregor and Y.~LeCun}, {\em Learning fast approximations of sparse
  coding}, in Proceedings of the 27th international conference on international
  conference on machine learning, 2010, pp.~399--406.

\bibitem{jiao2017iterative}
{\sc Y.~Jiao, B.~Jin, and X.~Lu}, {\em Iterative soft/hard thresholding with
  homotopy continuation for sparse recovery}, IEEE Signal Processing Letters,
  24 (2017), pp.~784--788.

\bibitem{kingma2014adam}
{\sc D.~P. Kingma and J.~Ba}, {\em Adam: A method for stochastic optimization},
  arXiv preprint arXiv:1412.6980,  (2014).

\bibitem{kingma2013auto}
{\sc D.~P. Kingma and M.~Welling}, {\em Auto-encoding variational bayes}, arXiv
  preprint arXiv:1312.6114,  (2013).

\bibitem{liao2014generalized}
{\sc X.~Liao, H.~Li, and L.~Carin}, {\em Generalized alternating projection for
  weighted-$\ell_{2,1}$ minimization with applications to model-based
  compressive sensing}, SIAM Journal on Imaging Sciences, 7 (2014),
  pp.~797--823.

\bibitem{liu2019alista}
{\sc J.~Liu, X.~Chen, Z.~Wang, and W.~Yin}, {\em {ALISTA}: Analytic weights are
  as good as learned weights in {LISTA}}, in International Conference on
  Learning Representations (ICLR), 2019.

\bibitem{moreau1965proximite}
{\sc J.-J. Moreau}, {\em Proximit{\'e} et dualit{\'e} dans un espace
  hilbertien}, Bull. Soc. Math. France, 93 (1965), pp.~273--299.

\bibitem{oymak2017sharp}
{\sc S.~Oymak, B.~Recht, and M.~Soltanolkotabi}, {\em Sharp time--data
  tradeoffs for linear inverse problems}, IEEE Transactions on Information
  Theory, 64 (2017), pp.~4129--4158.

\bibitem{plan2012robust}
{\sc Y.~Plan and R.~Vershynin}, {\em Robust 1-bit compressed sensing and sparse
  logistic regression: A convex programming approach}, IEEE Transactions on
  Information Theory, 59 (2012), pp.~482--494.

\bibitem{radford2015unsupervised}
{\sc A.~Radford, L.~Metz, and S.~Chintala}, {\em Unsupervised representation
  learning with deep convolutional generative adversarial networks}, arXiv
  preprint arXiv:1511.06434,  (2015).

\bibitem{rudin1992nonlinear}
{\sc L.~I. Rudin, S.~Osher, and E.~Fatemi}, {\em Nonlinear total variation
  based noise removal algorithms}, Physica D: nonlinear phenomena, 60 (1992),
  pp.~259--268.

\bibitem{saad2003iterative}
{\sc Y.~Saad}, {\em Iterative methods for sparse linear systems}, vol.~82,
  SIAM, 2003.

\bibitem{sabulal2020joint}
{\sc A.~P. Sabulal and S.~Bhashyam}, {\em Joint sparse recovery using deep
  unfolding with application to massive random access}, in IEEE International
  Conference on Acoustics, Speech and Signal Processing (ICASSP), IEEE, 2020,
  pp.~5050--5054.

\bibitem{shah2018solving}
{\sc V.~Shah and C.~Hegde}, {\em Solving linear inverse problems using gan
  priors: An algorithm with provable guarantees}, in 2018 IEEE International
  Conference on Acoustics, Speech and Signal Processing (ICASSP), IEEE, 2018,
  pp.~4609--4613.

\bibitem{Stein:1981vf}
{\sc C.~M. Stein}, {\em Estimation of the mean of a multivariate normal
  distribution}, The annals of Statistics,  (1981), pp.~1135--1151.

\bibitem{sun2008image}
{\sc J.~Sun, Z.~Xu, and H.-Y. Shum}, {\em Image super-resolution using gradient
  profile prior}, in 2008 IEEE Conference on Computer Vision and Pattern
  Recognition, IEEE, 2008, pp.~1--8.

\bibitem{teodoro2018convergent}
{\sc A.~M. Teodoro, J.~M. Bioucas-Dias, and M.~A. Figueiredo}, {\em A
  convergent image fusion algorithm using scene-adapted gaussian-mixture-based
  denoising}, IEEE Transactions on Image Processing, 28 (2018), pp.~451--463.

\bibitem{tikhonov1963solution}
{\sc A.~N. Tikhonov}, {\em On the solution of ill-posed problems and the method
  of regularization}, in Doklady Akademii Nauk, vol.~151, Russian Academy of
  Sciences, 1963, pp.~501--504.

\bibitem{tirer2019image}
{\sc T.~Tirer and R.~Giryes}, {\em Image restoration by iterative denoising and
  backward projections}, IEEE Transactions on Image Processing, 28 (2018),
  pp.~1220--1234.

\bibitem{tirer2018icip}
{\sc T.~Tirer and R.~Giryes}, {\em An iterative denoising and backwards
  projections method and its advantages for blind deblurring}, in 2018 25th
  IEEE International Conference on Image Processing (ICIP), IEEE, 2018,
  pp.~973--977.

\bibitem{tirer2018super}
{\sc T.~Tirer and R.~Giryes}, {\em Super-resolution via image-adapted denoising
  {CNNs}: Incorporating external and internal learning}, IEEE Signal Processing
  Letters, 26 (2019), pp.~1080--1084.

\bibitem{tirer2019back}
{\sc T.~Tirer and R.~Giryes}, {\em Back-projection based fidelity term for
  ill-posed linear inverse problems}, IEEE Transactions on Image Processing, 29
  (2020), pp.~6164--6179.

\bibitem{wu2019sparse}
{\sc K.~Wu, Y.~Guo, Z.~Li, and C.~Zhang}, {\em Sparse coding with gated learned
  {ISTA}}, in International Conference on Learning Representations, 2019.

\bibitem{yang2010image}
{\sc J.~Yang, J.~Wright, T.~S. Huang, and Y.~Ma}, {\em Image super-resolution
  via sparse representation}, IEEE transactions on image processing, 19 (2010),
  pp.~2861--2873.

\bibitem{yogev2021interpretation}
{\sc E.~Yogev-Ofer, T.~Tirer, and R.~Giryes}, {\em An interpretation of
  regularization by denoising and its application with the back-projected
  fidelity term}, arXiv preprint arXiv:2101.11599,  (2021).

\bibitem{yuan2021plug}
{\sc X.~Yuan, Y.~Liu, J.~Suo, F.~Durand, and Q.~Dai}, {\em Plug-and-play
  algorithms for video snapshot compressive imaging}, arXiv preprint
  arXiv:2101.04822,  (2021).

\bibitem{zarka2019deep}
{\sc J.~Zarka, L.~Thiry, T.~Angles, and S.~Mallat}, {\em Deep network
  classification by scattering and homotopy dictionary learning}, in
  International Conference on Learning Representations, 2019.

\bibitem{zhang2017beyond}
{\sc K.~Zhang, W.~Zuo, Y.~Chen, D.~Meng, and L.~Zhang}, {\em Beyond a gaussian
  denoiser: Residual learning of deep cnn for image denoising}, IEEE
  Transactions on Image Processing, 26 (2017), pp.~3142--3155.

\bibitem{zukerman2020bp}
{\sc J.~Zukerman, T.~Tirer, and R.~Giryes}, {\em {BP-DIP}: A backprojection
  based deep image prior}, 2020 28th European Signal Processing Conference
  (EUSIPCO),  (2020), pp.~675--679.

\end{thebibliography}
